\documentclass[11pt,a4paper]{article}
\usepackage{color}
\usepackage{graphicx}
\usepackage{amsfonts}
\usepackage{amsmath,amsthm,amssymb,color}
\usepackage{hyperref}%参考文献互相显示
\usepackage{paralist}
\usepackage{enumitem}
\usepackage{eepic}
\theoremstyle{theorem}
\newtheorem{theorem}{Theorem}[section]

\newtheorem{lemma}{Lemma}[section]
\newtheorem{corollary}{Corollary}[section]

\theoremstyle{definition}

\newtheorem{claim}{Claim}[section]

\newtheorem{conjecture}{Conjecture}[section]

\newtheorem{observation}{Observation}[section]

\begin{document}
\title{\bf The Tur\'an problems of directed paths and cycles in digraphs
\thanks{Supported by NSFC (Nos. 12071370, 11871311, 11631014).}}
\date{}
\author{Wenling Zhou \thanks{School of Mathematics,
Shandong University, Jinan 250100, P.R. China. Email:
\texttt{gracezhou@mail.sdu.edu.cn}.} \and Binlong
Li\thanks{Corresponding author. School of Mathematics and
Statistics, Northwestern Polytechnical University, Xi'an 710072,
P.R. China. Email: \texttt{libinlong@mail.nwpu.edu.cn}.}}
%\author{Wenling Zhou$^1$, Binlong Li$^2$\\[2mm]
%\small $^1$ School of Mathematics, Shandong University, Jinan 250100, P.R. China\\
%\small $^2$ School of Mathematics and Statistics, Northwestern
%Polytechnical University, Xi'an 710072, P.R. China}
\maketitle

\begin{center}
\begin{minipage}{130mm}
\small\noindent{\bf Abstract:}
%Given a digraph $D$, let $\dex(n,D)$
%denote the maximum number of arcs in a $n$-vertex digraph that
%contains no $D$ as a subdigraph. A $D$-free digraph on $n$ vertices
%is extremal if it has $\dex(n,D)$ arcs. A directed path (or cycle)
%is one in which the edges are all oriented in the same direction.
Let $\overrightarrow{P_k}$ and $\overrightarrow{C_k}$ denote the
directed path and the directed cycle of order $k$, respectively. In
this paper, we determine the precise maximum size of
$\overrightarrow{P_k}$-free digraphs of order $n$ as well as the
extremal digraphs attaining the maximum size for large $n$. For all
$n$, we also determine the precise maximum size of
$\overrightarrow{C_k}$-free digraphs of order $n$ as well as the
extremal digraphs attaining the maximum size. In addition, Huang and
Lyu [\textit{Discrete Math. 343(5) 2020}] characterized the extremal
digraphs avoiding an orientation of $C_4$. For all other
orientations of $C_4$, we also study the maximum size and the
extremal digraphs avoiding them.

%we determine $\dex(n,\overrightarrow{P_k})$ and all extremal digraphs for
%$\overrightarrow{P_k}$ for large $n$. However, if $n$ is small, then the
%underlying multigraphs of the extremal digraphs are not unique.

\end{minipage}
\end{center}

\smallskip

%---------------------
\section{Introduction}
%---------------------

The Tur\'an problem is a classic extremal problem in graph theory,
whose origin can be traced back to Tur\'an's generalization of
Mantel's theorem~\cite{Mantel1907}. How many edges guarantee that a
graph on $n$ vertices has a complete subgraph with $k$ vertices, no
matter how these edges are arranged? The answer is given by
Tur\'an's theorem~\cite{Tu1941} in 1941. Tur\'an's theorem is not
merely one extremal result among others, but also initiated extremal
graph theory. In general, for a given graph (or digraph) $H$, a
graph (or digraph) is called \textit{$H$-free} if it contains no $H$
as a subgraph (or subdigraph). An $H$-free graph (or digraph) on $n$
vertices with the maximum possible number of edges (or arcs) is
called \textit{extremal} for $H$ and $n$; its number of edges (or
arcs) is denoted by ${\rm ex}(n,H)$. %(or $\dex(n,H)$)
This quantity is referred to as the \textit{Tur\'an number} for $H$
and $n$. Sometimes, the $H$-free extremal graphs may be not unique,
we define ${\rm EX}(n,H)$ to be the family of all  $H$-free extremal
graphs (or digraphs). The Tur\'an problem focuses on the Tur\'an
number and the extremal graphs (or digraphs) for given $H$. Most
results on classical Tur\'an problem have been obtained for
undirected graphs,
see~\cite{bollobas1996extremal,chung1999upper,MR3822066,sidorenko1995we,jukna2011extremal}.
Relatively few results for directed graphs have been obtained,
see~\cite{brown1973extremal,brown1970extremal,
brown2002extremal,huang1608turan,huang2011digraphs}. In this paper,
we are interested in the Tur\'an problem for directed paths and
directed cycles.

The Tur\'an problem of complete digraphs and tournaments has been
solved in work of Brown and Harary~\cite{brown1970extremal}, and the
extremal digraphs have a interestingly counterpart to the extremal
graphs for complete graphs (see in \cite{brown1970extremal}). They
also studied the Tur\'an problem of some digraphs on at most $4$
vertices where any two vertices are joined by at least one arc.
%Note particularly the extremal digraphs are more complex avoiding a complete digraph (or tournament) with fewer vertices.
%In addition, by using matrices, Jacob and Meyniel~\cite{1983Extension} also studied the Tur\'an problem of complete digraphs.
The study of paths and cycles are also the focus
of the Tur\'an problem.
Erd\H{o}s and Gallai~\cite{erdHos1959maximal} determined the Tur\'an number of paths. Later, Faudree and Schelp~\cite{1975Path} gave a completely characterization of
the extremal graphs for paths. However, it is very difficult to
determine the exact Tur\'an number of the even cycles;
see~\cite{ball2012asymptotic,bukh2017bound,firke2013extremal,
lazebnik1994properties,pikhurko2012note}. Even for $C_4$ (the cycle
of order 4), we only know some asymptotic results. Therefore, it is
interesting to study the Tur\'an problem of a specific orientation
of paths (or cycles). Among all the orientations of paths (or
cycles), the directed path (or directed cycle) is one of the most
natural orientation. In~\cite{1982On, 1983On}, Howalla et al.
determined the maximum size of the digraphs without 2-cycles and
avoiding $k$ directed paths with the same initial vertex and
terminal vertex for $k=2,3$.
%Maurer-Rabinovitch-Trotter~\cite{1980A} studied the extremal $\overrightarrow{C_2}$-free digraphs transitive digraphs which contain at most one directed path from $x$ to $y$ for any two distinct vertices $x$, $y$.
%In~\cite{MR3892357,MR3922460,MR2737971,MR2639253}, the authors
%studied the extremal digraphs avoiding distinct walks of a given length $k$ with the same initial vertex and the same terminal vertex.
By using matrices, Huang and Lyu~\cite{MR4058215} studied the
Tur\'an problem of a specific orientation of $C_4$, denoted by
$\overrightarrow{P_{2,2}}$, which consists of two paths of order 3
with the same initial and terminal vertices.

Inspired by the studied of Brown-Harary and Huang-Lyu, we naturally
want to know if there are interesting relationships between the
extremal structures of directed paths (or directed cycles) and the
extremal structures of paths and cycles. In addition, a natural
important question: is there a kind of orientation of $C_4$ such
that the Tur\'an number is hard as difficult to calculate as $C_4$?
In this paper, we will answer these questions and shed light on the
relationship of extremal structures between graphs and digraphs.

%In this paper, we determine the exact
%Tur\'an number of directed paths and directed cycles. Finally, we
%find that the extremal structures of directed paths is closely
%related to some extremal structures in undirected graphs.

\subsection{Preliminaries}
In order to present our results, we begin with some notation. We also introduce some classical results that will be used in the proofs.

To avoid confusion, we first distinguish some notation between
graphs and digraphs. A (simple) \textit{graph}, is a pair
$G=(V(G),E(G))$, where $V(G)$ is a non-empty finite \textit{vertex}
set and $E(G)$, a set of $2$-element subsets of $V(G)$ is an
\textit{edge} set. A \textit{directed graph, digraph} for short, is
a pair $D=(V(D), A(D))$, where $V(D)$ is also a \textit{vertex} set,
but $A(D)$ is a set of order pairs of vertices, called \textit{arc}
(or \textit{directed edge}) set. The cardinality of $E(G)$ and
$A(D)$ is denoted by $e(G)$ and $a(D)$, respectively. An edge (or
arc) with identical ends is called a \textit{loop}. Two or more
edges with the same pair of ends are said to be \textit{parallel}
edges. If two arcs between a pair of vertices have the same
direction, we say they are \textit{parallel}. In this paper, our
notation depends heavily on the use of letters $G$ for graphs and
$D$ for digraphs, and graphs and digraphs are strict, i.e., they do
not allow loops and parallel edges (arcs). {If a graph contains
parallel edges, then we call it a \textit{multigraph}. Multigraphs
in this paper have maximum edge-multiplicity $2$.} In particular, we
say that $G$ (or $D$) is \textit{complete} if $E(G)=\{\{u,v\}:
u,v\in V(G), u\neq v\}$ (or $A(D)=\{(u,v): u,v\in V(D), u\neq v\}$).
We denote by $K_k$ (or $\overleftrightarrow{K_k}$) the complete
graph (or complete digraph) with $k$ vertices. A subset of $V(G)$
(or $V(D)$) is \textit{independent} in $G$ (or $D$) if no two of its
vertices are connected by an edge (or arc).

An \textit{orientation} of a graph (or multigraph) $G$ is a digraph
obtained from $G$ by assigning a direction to each edge of $G$. {In
this paper, the orientations of multigraphs are strict, i.e., they
have no parallel arcs.} In general, an orientation of a complete
graph $K_k$ is called a \textit{tournament}, denoted by
$\overrightarrow{T_k}$.
%A tournament is \textit{transitive} if it has no directed cycle.
A \emph{directed} path (or \emph{directed} cycle) is an orientation
of a path (or cycle) such that the direction of each arc is same to
that of its predecessor. For convenience, we denote by $P_k$ (or
$\overrightarrow{P_k}$) the path (or directed path) and $C_k$ (or
$\overrightarrow{C_k}$) the cycle (or directed cycle) with $k$
vertices. For a digraph $D$, the \textit{underlying multigraph}
{${\rm Umg}(D)$} of $D$ is a multigraph obtained from $D$ by
forgetting the direction on each arc of $D$. The \textit{underlying
graph} {${\rm Ung}(D)$} of $D$ is obtained from ${\rm Umg}(D)$ by
replacing every pair of parallel edges with one. For example, for a
digraph $D$ with vertices $u$, $v$ and arcs $(u,v)$, $(v,u)$, ${\rm
Umg}(D)$ has two parallel edges between $u,v$ and ${\rm Ung}(D)$ has
only one. Accordingly, for a graph $G$, the \textit{double
orientation} $\overleftrightarrow{G}$ of $G$ is a digraph obtained
from $G$ by replacing every edge $\{u,v\}$ of $G$ with two arcs
$(u,v)$ and $(v,u)$. Note that ${\rm Ung}(\overleftrightarrow{G})=G$
and $a(\overleftrightarrow{G})=2e(G)$.

As usual, in graphs we use $uv$ to denote the edge $\{u,v\}$ for
short, and in digraphs we use $uv$ to denote the arc $(u,v)$. For a
digraph $D=(V,A)$ with $u,v\in V$, if $uv\in A$, we
%say that $u$ \textit{dominates} $v$ ($v$ is \textit{dominated} by $u$) and
denote it by $u\rightarrow v$. If $uv\notin A$, we denote it by
$u\nrightarrow v$. Furthermore, if $uv\in A$ and $vu\notin A$, we
denote it by $u\mapsto v$.
%In this paper, the notation $u\mapsto v$ means that $(u,v)\in A$ and $(v,u)\notin A$, and $u\leftrightarrow v$ means that $(u,v) \in A$ and $(v,u) \in A$.
Analogously, for two subset $U,W\subseteq V$, we write $U\rightarrow
W$ if $u\rightarrow v$ for every $u\in U$ and $v\in W$, and write
$U\nrightarrow W$ if $u\nrightarrow v$ for every $u\in U$ and $v\in
W$. The notation $U\mapsto W$ means that $U\rightarrow W$ and
$W\nrightarrow U$. If $U$ (or $W$) consists of a single vertex $x$,
then we simply write $x\rightarrow W$, $x\nrightarrow W$ and
$x\mapsto W$ (or $U\rightarrow x$, $U\nrightarrow x$, and $U\mapsto
x$).

%The tournament $D$ is \textit{transitive} if for each three distinct vertices $x,y,z\in
%V(D)$, $x\rightarrow y$ and $y\rightarrow z$ implies $x\rightarrow
%z$ in $D$. For a digraph $D$, the \textit{complement} $\bar{D}$ of $D$ is the digraph on $V(D)$ such that $u\rightarrow v$ in $\bar{D}$ if and only if $u\nrightarrow v$
%in $D$.

Let $k\geq 1$ be an integer. A graph $G$ is called
\textit{$k$-partite} if $V(G)$ admits a partition into $k$ classes
such that every edge has its ends in different classes. A
$k$-partite graph is called \textit{complete $k$-partite} if its any
two vertices from different partition sets are adjacent. We denote
by $K_{n_1,n_2,\ldots,n_k}$ the complete $k$-partite graph with
partition sets of size $n_1,n_2,\ldots,n_k$. A partition of a vertex
set is an \textit{almost balanced partition} if its partition sets
differ in size by at most $1$. Thus, a complete graph $K_k$ can be
thought as a complete $k$-partite graph such that each partition set
has exactly one vertex. Denote by $K^s_k$ the complete $k$-partite
graph whose each partition set has exactly $s$ vertices.

For a directed path $\overrightarrow{P}$ and $u,v\in
V(\overrightarrow{P})$, we use $\overrightarrow{P}[u,v]$ to denote
the directed sub-path from $u$ to $v$. For a directed cycle
$\overrightarrow{C}$ and $u,v\in V(\overrightarrow{C})$, we use
$\overrightarrow{C}[u,v]$ to denote the directed path from $u$ to
$v$ along $\overrightarrow{C}$.

We use $G_1\cup G_2$ to denote the \emph{disjoint union} of the two
graphs $G_1,G_2$; and $G_1\vee G_2$ to denote the \emph{join} of
$G_1,G_2$, that is, the graphs with vertex set $V(G_1)\cup V(G_2)$
and edge set $E(G_1)\cup E(G_2)\cup\{uv: u\in V(G_1), v\in
V(G_2)\}$. The disjoint union of $k$ copies of $G$ is denoted by
$kG$, and the \emph{complement} of $G$, denoted by $\overline{G}$,
is the graphs on $V(G)$ such that for two vertices $u,v\in V(G)$,
$uv\in E(\overline{G})$ if and only if $uv\notin E(G)$.

Tur\'an's theorem~\cite{jukna2011extremal} gave the unique
$K_{k+1}$-free extremal graph of order $n$, which we shall denote by
$T_{n,k}$ and call it the \emph{Tur\'an graph.} The Tur\'an graph
$T_{n,k}$ is a complete $k$-partite graph with an almost balanced
partition on $n$ vertices.

\begin{theorem}[Tur\'an \cite{jukna2011extremal}]\label{K_k-free}
Let $n, k\in\mathbb{N}^*$, $n=qk+r,\ 0\leq r<k$.
Then
$${\rm ex}(n,K_{k+1})=e(T_{n,k})=\frac{k-1}{2k}n^2-\frac{r(k-r)}{2k},$$
and ${\rm EX}(n,K_{k+1})=\{T_{n,k}\}$.
\end{theorem}

Further, for any given $K^s_{k}$, the famous Erd\H{o}s-Stone theorem
gave the upper bound of Tur\'an number ${\rm ex}(n,K^s_{k})$.

\begin{theorem}[Erd\H{o}s, Stone \cite{MR3822066}]\label{K^s_k-free}
For all $k, s\in\mathbb{N}^*$, and any $\varepsilon > 0$, there
exists an integer $n_0$ such that for all $n\geq n_0$, we have
$${\rm ex}(n,K^s_{k+1})<e(T_{n,k})+\varepsilon n^2.$$
\end{theorem}

%Erd\H{o}s and Gallai~\cite{erdHos1959maximal} determined the Tur\'an number ${\rm
%ex}(n,P_{k+1})$ and the $P_{k+1}$-free extremal graphs for $k,n
%\in\mathbb{N}^*$ with $k|n$. Later, Faudree and
%Schelp~\cite{1975Path} completely determined ${\rm ex}(n,P_{k+1})$
%and ${\rm EX}(n,P_{k+1})$.
Let $n, k\in\mathbb{N}^*$ with $n=qk+r$ and $0\leq r<k$. In this
paper, We set $\varGamma_{n,k}:=qK_k\cup K_r$, that is, the the
graph consists of $q$ copies of $K_k$ and one copy of $K_r$. In
addition, if $k\ge 3$ is odd, $r = (k \pm 1)/2$ and $0\le\ell\leq
q$,  we define the graph $\varGamma_{n,k,\ell}:=(q-\ell)K_k\cup
(K_{(k-1)/2}\vee\overline{K_{\ell k+r-(k-1)/2}})$. Note that
$\varGamma_{n,k,0}=\varGamma_{n,k}$.

\begin{theorem}[Faudree, Schelp~\cite{1975Path}]\label{E-G-path}
Let $n, k\in\mathbb{N}^*$, $n=qk+r,\ 0\leq r<k$. Then
\[
{\rm
ex}(n,P_{k+1})=e(\varGamma_{n,k})=\frac{k-1}{2}n-\frac{r(k-r)}{2},
\]
and ${\rm EX}(n,P_{k+1})=\{\varGamma_{n,k}\}$ except for $k\ge 3$ is
odd and $r = (k \pm 1)/2$, in which case ${\rm
EX}(n,P_{k+1})=\{\varGamma_{n,k,\ell}: 0 \le \ell \leq q\}$.
\end{theorem}

A path (or cycle) of a graph $G$ or a directed path (or directed
cycle) of a digraph $D$ is called \textit{Hamilton} one if it
contains every vertex of $G$ or $D$. We will also use the following
classic result in this paper.

\begin{theorem}[R\'{e}dei \cite{bang2008digraphs}]\label{H-path}
Every tournament contains a Hamilton directed path.
\end{theorem}

\subsection{Our results}

\subsubsection{The Tur\'an problem of directed paths}

For integers $k,\ell$, we denote by $[k,\ell]$ the set of integers
$i$ with $k\leq i\leq \ell$. For $n, k\in\mathbb{N}^*$ with $n\geq
k$, we define a class of digraphs on $n$ vertices, denoted by
$\overrightarrow{T_{n,k}}$, which satisfies the following
properties:
\begin{compactenum}[\rm (i)]
  \item $V(\overrightarrow{T_{n,k}})$ has an almost balanced partition $\{V_1,V_2,\ldots,V_k\}$;\label{path-item:a}
  \item $V_i$ is an independent set of $\overrightarrow{T_{n,k}}$ for $i\in [1,k]$;  \label{path-item:b}
  \item for any two parts $V_i$ and $V_j$ with $1\le i< j\le k$,
$V_i\mapsto V_j$ (see
Figure~1).\label{path-item:c}
\end{compactenum}
%Furthermore, we also define a digraph on $n$ vertices, denoted by
%$\ria{\varGamma}_{n,k}$, which consists of $q$ copies of $\overleftrightarrow{K}_k$
%and one copy of $\overleftrightarrow{K}_r$, where $n=qk+r,\ 0\leq r<k$ .

\noindent {\bf Remark:} Note that there may be more than one
non-isomorphic $\overrightarrow{T_{n,k}}$ for given $n,k$. However
all the $\overrightarrow{T_{n,k}}$ have the same arc number. We
denote by $\overrightarrow{\mathcal{T}_{n,k}}$ the set of all
$\overrightarrow{T_{n,k}}$. Observe that the underlying multigraph
of $\overrightarrow{T_{n,k}}$ is the Tur\'an graph $T_{n,k}$. We
sometimes call $\overrightarrow{T_{n,k}}$ a \emph{transitive
orientation} of $T_{n,k}$.
%, and the underlying graph of $\ria{\varGamma}_{n,k}$ is
%the graph $\varGamma_{n,k}$.

\begin{center}
\begin{picture}(200,90)
\thicklines \put(0,0){\put(30,40){\circle{24}}
\put(70,40){\circle{24}} \put(150,40){\circle{24}}
\put(190,40){\circle{24}} \put(103,40){\circle*{3}}
\put(110,40){\circle*{3}} \put(117,40){\circle*{3}}
\multiput(30,40)(40,0){4}{\put(12,0){\line(1,0){16}}
\put(24,0){\vector(1,0){0}}}
\multiput(30,40)(40,0){3}{\qbezier(10.4,6)(40,26)(69.6,6)
\put(44,16){\vector(1,0){0}}}
\multiput(30,40)(40,0){2}{\qbezier(6,10.4)(60,40)(114,10.4)
\put(64,25){\vector(1,0){0}}} \qbezier(30,52)(110,100)(190,52)
\put(114,76){\vector(1,0){0}} \put(25,35){$V_1$} \put(65,35){$V_2$}
\put(139,35){$V_{k-1}$} \put(185,35){$V_k$} \put(60,10){Figure 1.
Digraph $\overrightarrow{T_{n,k}}$} }
%\put(210,0){\put(30,40){\circle{24}} \put(70,40){\circle{24}}
%\put(130,40){\circle{24}} \put(170,40){\circle{24}}
%\put(93,40){\circle*{3}} \put(100,40){\circle*{3}}
%\put(107,40){\circle*{3}} \put(25,35){$\overleftrightarrow {K}_k$}
%\put(65,35){$\overleftrightarrow {K}_k$} \put(125,35){$\overleftrightarrow {K}_k$}
%\put(165,35){$\overleftrightarrow {K}_r$} \put(60,70){$q$ copies}
%\put(25,55){$\overbrace{\ \ \ \ \ \ \ \ \ \ \ \ \ \ \ \ \ \ \ \ \ \
%\ \ \ \ \ \ \ }$} \put(90,10){$\ria{\varGamma}_{n,k}$}}
\end{picture}

%\small Figure 1. Digraph $\overrightarrow{T_{n,k}}$.
 %and $\ria{\varGamma}_{n,k}$.
\end{center}

Now we deal with the Tur\'an Problem for directed paths
$\overrightarrow{P_{k+1}}$. For $k=1$, the
$\overrightarrow{P_2}$-free digraph is empty, and thus ${\rm
ex}(n,\overrightarrow{P_2})=0$. For $k=2$, clearly ${\rm
ex}(1,\overrightarrow{P_3})=0$, ${\rm
ex}(2,\overrightarrow{P_3})=2$, and when $n\geq 3$, we have the
following result.

\begin{theorem}\label{3-path}
For all $n\geq 3$,
\[
{\rm
ex}(n,\overrightarrow{P_3})=a(\overrightarrow{T_{n,2}})=\left\lfloor{\frac{n^2}{4}}\right\rfloor,
\]
and ${\rm
EX}(n,\overrightarrow{P_3})=\overrightarrow{\mathcal{T}_{n,2}}$ for
$n\ge 5$; and ${\rm
EX}(n,\overrightarrow{P_3})=\overrightarrow{\mathcal{T}_{n,2}}\cup\{\overleftrightarrow{\varGamma_{n,2}}\}$
for $n=3,4$.
\end{theorem}

However, for $k\geq 3$ and $n$ is small, it is difficult to
determine the precise Tur\'an number ${\rm
ex}(n,\overrightarrow{P_{k+1}})$. With more effort, we prove the
following result.

\begin{theorem} \label{P_k}
For every $k\in \mathbb{N}^*$, there exists $n_0$ such that if
$n\geq n_0$, $n=qk+r,\ 0\leq r<k$, then
\[
{\rm ex}
(n,\overrightarrow{P_{k+1}})=a(\overrightarrow{T_{n,k}})=\frac{k-1}{2k}n^2-\frac{r(k-r)}{2k},
\]
and ${\rm
EX}(n,\overrightarrow{P_{k+1}})=\overrightarrow{\mathcal{T}_{n,k}}$.
\end{theorem}

\begin{corollary} \label{P_k-upper-bound}
For all $k,n\in \mathbb{N}^*$,
$${\rm ex}(n,\overrightarrow{P_{k+1}})\leq{\rm ex}(n,K_{k+1})+{\rm ex}(n,P_{k+1}).$$
\end{corollary}

\subsubsection{The Tur\'an problem of directed cycles }
For integers $n,k\in \mathbb{N}^*$, $n=qk+r,\ 0\leq r<k$, we define
a class of digraphs on $n$ vertices, denoted by
$\overrightarrow{F_{n,k}}$, which satisfies the following properties
(set $q'=q$ if $r=0$ and $q'=q+1$ otherwise):
\begin{compactenum}[\rm (i)]
  \item $V(\overrightarrow{F_{n,k}})$ has a partition $\{V_1,V_2,\ldots,V_{q'}\}$ such that $|V_i|=k$ for all but at most one $V_i$; \label{directed path-item:a}
  \item $\overrightarrow{F_{n,k}}[V_i]$ is a complete digraph for $i\in [1,q']$;  \label{directed path-item:b}
  \item for any two parts $V_i$ and $V_j$ with $1\le i< j\le q'$,
$V_i\mapsto V_j$ (see Figure~2).\label{directed path-item:c}
\end{compactenum}
Note that there may be more than one non-isomorphic
$\overrightarrow{F_{n,k}}$ for given $n,k$. However all
$\overrightarrow{F_{n,k}}$ have the same arc number. We denote by
$\overrightarrow{\mathcal{F}_{n,k}}$ the set of all
$\overrightarrow{F_{n,k}}$.

\begin{center}
\begin{picture}(220,90)
\thicklines \put(0,5){\put(30,40){\circle{24}}
\put(70,40){\circle{24}} \put(150,40){\circle{24}}
\put(190,40){\circle{24}} \put(103,40){\circle*{3}}
\put(110,40){\circle*{3}} \put(117,40){\circle*{3}}
\multiput(30,40)(40,0){4}{\put(12,0){\line(1,0){16}}
\put(24,0){\vector(1,0){0}}}
\multiput(30,40)(40,0){3}{\qbezier(10.4,6)(40,26)(69.6,6)
\put(44,16){\vector(1,0){0}}}
\multiput(30,40)(40,0){2}{\qbezier(6,10.4)(60,40)(114,10.4)
\put(64,25){\vector(1,0){0}}} \qbezier(30,52)(110,100)(190,52)
\put(114,76){\vector(1,0){0}}
\put(23,35){$\overleftrightarrow{K_k}$}
\put(63,35){$\overleftrightarrow{K_k}$}
\put(143,35){$\overleftrightarrow{K_r}$}
\put(183,35){$\overleftrightarrow{K_k}$}}

\put(25,30){$\underbrace{\ \ \ \ \ \ \ \ \ \ \ \ \ \ \ \ \ \ \ \ \ \
\ \ \ \ \ \ \ \ \ \ \ \ \ \ \ \ \ \ \ \ \ \ \ \ \ }$}
\put(30,10){$q$ copies of $\overleftrightarrow{K_k}$ and one copy of
$\overleftrightarrow{K_r}$}
\end{picture}

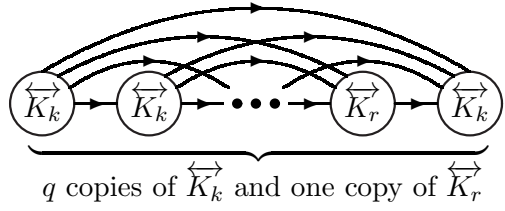
\small Figure~2. Digraph $\overrightarrow{F_{n,k}}$.
\end{center}

Now we deal with the Tur\'an problem for directed cycles
$\overrightarrow{C_k}$. First, for $k\in\{2,3\}$,
$\overrightarrow{C_2}$ can be regarded as a complete digraph
$\overleftrightarrow{K_2}$, and $\overrightarrow{C_3}$ can be
regarded as a tournament $\overrightarrow{T_3}$. For complete
digraph $\overleftrightarrow{K_k}$ and tournament
$\overrightarrow{T_k}$, let us recall the results of Brown and
Harary~\cite{brown1970extremal}.
%In order to present the results of Brown and Harary, we need the
%following notation and definitions. Given two graph $G$ and $G'$ on
%the same vertex set $V$, the \textit{union} $G \cup G'$ is defined
%as follows: $V(G \cup G')=V$ and $E(G \cup G')=E(G)\cup E(G')$.
%According to the graph $T_{n,k}$, we define a digraph
%$\overleftrightarrow{T}(n,k)$ by replacing every edge $\{u,v\}$ of
%$T_{n,k}$ with two arcs $u\rightarrow v$ and $v\rightarrow u$, i.e.,
%the underlying multigraph of $\overleftrightarrow{T}(n,k)$ is the
%union $T_{n,k}\cup T_{n,k}$.

\begin{theorem}[Brown, Harary~\cite{brown1970extremal}]\label{Brown-Harary}
Let $k,n\in \mathbb{N}^*$, then
\[
{\rm ex}(n,\overleftrightarrow{K_{k+1}})=\binom{n}{2}+{\rm ex}(n,
K_{k+1}) \mbox{ and \ } {\rm ex}(n,\overrightarrow{T_{k+1}})=2 {\rm
ex}(n, K_{k+1}).
\]
\end{theorem}
%\begin{theorem}[Brown-Harary~\cite{brown1970extremal}]\label{Brown-Harary}
%Let $k,n\in \mathbb{N}^*$, then
%\begin{compactenum}[\rm (i)]
%  \item $\ria{ex}(n,\overleftrightarrow{K}_k)=\binom{n}{2}+ ex(n, K_k)$, and ${\rm EX}(n,\overleftrightarrow{K}_k)=\{\text{all~orientations~of~the~union~} K_n \cup T(n,k-1)\}$.
%  \item $\ria{ex}(n,\overrightarrow{T_k})=2ex(n, K_k)$, and for $k\neq 3$, ${\rm EX}(n,\overleftrightarrow{K}_k)=\{\overleftrightarrow{T}(n,k-1)\}$; for $k=3$, ${\rm EX}(n,\overleftrightarrow{K}_3)$ consists of all possible digraphs of the form $\ria{T}^*:=\overleftrightarrow{T}(n_1,2)\mapsto \overleftrightarrow{T}(n_2,2)\mapsto  \ldots \mapsto \overleftrightarrow{T}(n_r,2),$
%     where \begin{enumerate}[label=$(\alph*)$]
%    \item $n=n_1+n_2+\ldots+n_r$ is any partition of $n$ into positive integers, no two of them odd $(r\ge 1)$;
%    \item $\ria{T}^*$ consists of the union $\overleftrightarrow{T}(n_1,2)\cup \overleftrightarrow{T}(n_2,2)\cup \ldots \cup \overleftrightarrow{T}(n_r,2)$ and all arcs of $V(\overleftrightarrow{T}(n_i,2))\mapsto V(\overleftrightarrow{T}(n_j,2))$ for $1\le i<j\le r$.
%\end{enumerate}
%\end{compactenum}
%\end{theorem}
Moreover, Brown and Harary characterized all the extremal digraphs
for $\overleftrightarrow{K_k}$ and $\overrightarrow{T_k}$.
Therefore, the Tur\'an number and the extremal digraphs for
$\overrightarrow{C_2}$ and $\overrightarrow{C_3}$ can be known
directly by the results of Brown and Harary. For any directed cycles
$\overrightarrow{C_{k+1}}$, $k\geq 3$, we have the following
result.

\begin{theorem} \label{C_k}
Let $k,n\in \mathbb{N}^*$, $n=qk+r$, $0\leq r<k$. Then
\[
{\rm
ex}(n,\overrightarrow{C_{k+1}})=a(\overrightarrow{F_{n,k}})=\frac{1}{2}n^2+\frac{k-2}{2}n-\frac{r(k-r)}{2}.
\]
Furthermore, ${\rm
EX}(n,\overrightarrow{C_{k+1}})=\overrightarrow{\mathcal{F}_{n,k}}$
for $k\geq 3$.
\end{theorem}

It is worth mentioning that for $\overrightarrow{C_2}$ and
$\overrightarrow{C_3}$, there are extremal graphs other than that in
$\overrightarrow{\mathcal{F}_{n,1}}$ and
$\overrightarrow{\mathcal{F}_{n,2}}$ (see~\cite{brown1970extremal}).
In addition, by Theorem~\ref{C_k}, we can easily obtain the
following corollary about Hamilton directed cycle.

\begin{corollary}\label{C_n}
Every digraph with $n\geq 2$ vertices and at least $n^2-2n+2$ arcs
contains a Hamilton directed cycle.
\end{corollary}

\subsubsection{The Tur\'an problem of orientations of $C_4$ }

Finally, we will complete the Tur\'an problem of all orientations of
$C_4$.
%We define $\ria{P}_{s,t}$ as the digraph that consists of two internally
%disjoint directed path $\ria{P}_{s+1}$ and $\ria{P}_{t+1}$ with the
%same initial and terminal vertices.
{There are four orientations for $C_4$, we denote by
$\overrightarrow{C_4}$, $\overrightarrow{K_{2,2}}$,
$\overrightarrow{P_{2,2}}$ and $\overrightarrow{P_{1,3}}$ (see
Figure~3).}

\begin{center}
\begin{picture}(220,60)
\thicklines \put(0,10){ \put(0,10){\circle*{3}}
\put(0,40){\circle*{3}} \put(30,40){\circle*{3}}
\put(30,10){\circle*{3}} \put(0,10){\line(1,0){30}}
\put(0,10){\line(0,1){30}} \put(30,40){\line(-1,0){30}}
\put(30,40){\line(0,-1){30}} \put(0,25){\vector(0,-1){5}}
\put(15,10){\vector(1,0){5}} \put(30,25){\vector(0,1){5}}
\put(15,40){\vector(-1,0){5}} \put(10,-7){$\overrightarrow{C_4}$}}

\put(60,10){ \put(0,10){\circle*{3}} \put(0,40){\circle*{3}}
\put(30,40){\circle*{3}} \put(30,10){\circle*{3}}
\textcolor{blue}{\put(0,10){\line(1,0){30}}}
\put(0,10){\line(0,1){30}}
\textcolor{blue}{\put(30,40){\line(-1,0){30}}}
\put(30,40){\line(0,-1){30}} \put(0,25){\vector(0,-1){5}}
\put(15,10){\vector(-1,0){5}} \put(30,25){\vector(0,1){5}}
\put(15,40){\vector(1,0){5}} \put(5,-7){$\overrightarrow{K_{2,2}}$}}

\put(120,10){ \put(0,10){\circle*{3}} \put(0,40){\circle*{3}}
\put(30,40){\circle*{3}} \put(30,10){\circle*{3}}
\textcolor{blue}{\put(0,10){\line(1,0){30}}
\put(0,10){\line(0,1){30}}} \put(30,40){\line(-1,0){30}}
\put(30,40){\line(0,-1){30}} \put(0,25){\vector(0,-1){5}}
\put(15,10){\vector(1,0){5}} \put(30,25){\vector(0,-1){5}}
\put(15,40){\vector(1,0){5}} \put(5,-7){$\overrightarrow{P_{2,2}}$}}

\put(180,10){ \put(0,10){\circle*{3}} \put(0,40){\circle*{3}}
\put(30,40){\circle*{3}} \put(30,10){\circle*{3}}
\put(0,10){\line(1,0){30}}
\textcolor{blue}{\put(0,10){\line(0,1){30}}}
\put(30,40){\line(-1,0){30}} \put(30,40){\line(0,-1){30}}
\put(0,25){\vector(0,-1){5}} \put(15,10){\vector(-1,0){5}}
\put(30,25){\vector(0,-1){5}} \put(15,40){\vector(1,0){5}}
\put(5,-7){$\overrightarrow{P_{1,3}}$}}
\end{picture}

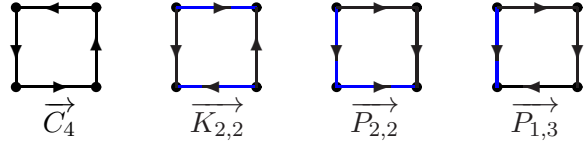
\small Figure~3. The orientations of $C_4$.
\end{center}

As mentioned earlier, the Tur\'an problem of
$\overrightarrow{P_{2,2}}$ has been solved by Huang and Lyu
in~\cite{MR4058215} (for $n\geq 13$), and the Tur\'an number of
$\overrightarrow{C_4}$ can be obtained directly from
Theorem~\ref{C_k}. Now we deal with the Tur\'an problem of
$\overrightarrow{P_{1,3}}$ and $\overrightarrow{K_{2,2}}$.

From the following Theorems \ref{p13small} and \ref{p13}, one can
see that ${\rm ex}(n,\overrightarrow{P_{1,3}})=\max\{2{\rm
ex}(n,C_4),{\rm ex}(n,K_4)\}$.

\begin{theorem} \label{p13small}
For $n\in[1,8]$, ${\rm ex}(n,\overrightarrow{P_{1,3}})=2{\rm ex}(n,
C_4)$, and ${\rm
EX}(n,\overrightarrow{P_{1,3}})=\{\overleftrightarrow{H}: H\in {\rm
EX}(n,C_4)\}$.
\end{theorem}

Computer searches by Clapham et al.~\cite{ClaphamFlockhartSheehan}
and Yang and Rowlinson~\cite{1992On} determined ${\rm EX}(n,C_4)$
for all $n\le 31$. We list the $C_4$-free extremal graphs for all
$n$ up to $8$ (see Figure~4).

\begin{center}
\begin{picture}(435,355)
\thicklines \put(0,0){\line(1,0){435}} \put(0,0){\line(0,1){355}}
\put(435,0){\line(0,1){355}} \put(0,105){\line(1,0){435}}
\put(0,190){\line(1,0){435}} \put(0,270){\line(1,0){435}}
\put(0,355){\line(1,0){435}}

\put(15,10){${\rm{EX}}(8,C_4)$, ${\rm{ex}}(8,C_4)=11$}

\put(20,30){\put(20,10){\circle*{4}} \put(20,40){\circle*{4}}
\put(45,0){\circle*{4}} \put(45,50){\circle*{4}}
\put(60,25){\circle*{4}} \put(0,25){\circle*{4}}
\put(75,0){\circle*{4}} \put(75,50){\circle*{4}}
\put(20,10){\line(5,-2){25}} \put(45,0){\line(3,5){15}}
\put(60,25){\line(-3,5){15}} \put(45,50){\line(-5,-2){25}}
\put(20,40){\line(0,-1){30}} \put(0,25){\line(4,-3){20}}
\put(0,25){\line(4,3){20}} \put(75,0){\line(-1,0){30}}
\put(75,0){\line(-3,5){15}} \put(75,50){\line(-1,0){30}}
\put(75,50){\line(-3,-5){15}} }

\put(115,30){\put(0,10){\circle*{4}} \put(0,40){\circle*{4}}
\put(25,0){\circle*{4}} \put(25,50){\circle*{4}}
\put(40,25){\circle*{4}} \put(55,0){\circle*{4}}
\put(55,50){\circle*{4}} \put(0,65){\circle*{4}}
\put(0,10){\line(5,-2){25}} \put(25,0){\line(3,5){15}}
\put(40,25){\line(-3,5){15}} \put(25,50){\line(-5,-2){25}}
\put(0,40){\line(0,-1){30}} \put(55,0){\line(-1,0){30}}
\put(55,0){\line(-3,5){15}} \put(55,50){\line(-1,0){30}}
\put(55,50){\line(-3,-5){15}} \put(0,65){\line(0,-1){30}}
\put(0,65){\line(5,-3){25}} }

\put(190,30){\put(0,25){\circle*{4}} \put(15,0){\circle*{4}}
\put(15,50){\circle*{4}} \put(40,10){\circle*{4}}
\put(40,40){\circle*{4}} \put(65,0){\circle*{4}}
\put(65,50){\circle*{4}} \put(80,25){\circle*{4}}

\put(0,25){\line(3,-5){15}} \put(0,25){\line(3,5){15}}
\put(15,0){\line(5,2){25}} \put(15,50){\line(5,-2){25}}
\put(40,10){\line(5,-2){25}} \put(65,0){\line(3,5){15}}
\put(80,25){\line(-3,5){15}} \put(65,50){\line(-5,-2){25}}
\put(40,40){\line(0,-1){30}} \put(15,0){\line(1,0){50}}
\put(15,50){\line(1,0){50}} }

\put(290,30){\put(0,0){\circle*{4}} \put(40,0){\circle*{4}}
\put(20,15){\circle*{4}} \put(0,30){\circle*{4}}
\put(20,32.5){\circle*{4}} \put(40,30){\circle*{4}}
\put(20,50){\circle*{4}} \put(60,15){\circle*{4}}
\put(0,0){\line(1,0){40}} \put(20,15){\line(-4,-3){20}}
\put(20,15){\line(4,-3){20}} \put(0,0){\line(0,1){30}}
\put(20,15){\line(0,1){35}} \put(40,0){\line(0,1){30}}
\put(20,50){\line(-1,-1){20}} \put(20,50){\line(1,-1){20}}
\put(40,0){\line(4,3){20}} \put(40,30){\line(4,-3){20}} }

\put(370,30){\put(0,0){\circle*{4}} \put(40,0){\circle*{4}}
\put(20,15){\circle*{4}} \put(0,30){\circle*{4}}
\put(20,32.5){\circle*{4}} \put(40,30){\circle*{4}}
\put(20,50){\circle*{4}} \put(45,55){\circle*{4}}
\put(0,0){\line(1,0){40}} \put(20,15){\line(-4,-3){20}}
\put(20,15){\line(4,-3){20}} \put(0,0){\line(0,1){30}}
\put(20,15){\line(0,1){35}} \put(40,0){\line(0,1){30}}
\put(20,50){\line(-1,-1){20}} \put(20,50){\line(1,-1){20}}
\put(40,30){\line(1,5){5}} \put(20,50){\line(5,1){25}} }

\put(15,115){${\rm{EX}}(7,C_4)$, ${\rm{ex}}(7,C_4)=9$}

\put(20,135){\put(0,20){\circle*{4}} \put(30,20){\circle*{4}}
\put(60,20){\circle*{4}} \put(15,0){\circle*{4}}
\put(45,0){\circle*{4}} \put(15,40){\circle*{4}}
\put(45,40){\circle*{4}} \put(15,40){\line(1,0){30}}
\put(0,20){\line(1,0){60}} \put(0,20){\line(3,-4){15}}
\put(15,40){\line(3,-4){30}} \put(45,0){\line(3,4){15}}
\put(15,0){\line(3,4){30}} }

\put(100,140){\put(0,0){\circle*{4}} \put(0,30){\circle*{4}}
\put(20,15){\circle*{4}} \put(50,15){\circle*{4}}
\put(70,0){\circle*{4}} \put(70,30){\circle*{4}}
\put(35,35){\circle*{4}} \put(0,0){\line(4,3){20}}
\put(0,0){\line(0,1){30}} \put(0,30){\line(4,-3){20}}
\put(20,15){\line(1,0){30}} \put(50,15){\line(4,3){20}}
\put(50,15){\line(4,-3){20}} \put(70,0){\line(0,1){30}}
\put(20,15){\line(3,4){15}} \put(50,15){\line(-3,4){15}} }

\put(190,130){\put(20,10){\circle*{4}} \put(20,40){\circle*{4}}
\put(45,0){\circle*{4}} \put(45,50){\circle*{4}}
\put(60,25){\circle*{4}} \put(0,25){\circle*{4}}
\put(75,0){\circle*{4}} \put(20,10){\line(5,-2){25}}
\put(45,0){\line(3,5){15}} \put(60,25){\line(-3,5){15}}
\put(45,50){\line(-5,-2){25}} \put(20,40){\line(0,-1){30}}
\put(0,25){\line(4,-3){20}} \put(0,25){\line(4,3){20}}
\put(75,0){\line(-1,0){30}} \put(75,0){\line(-3,5){15}} }

\put(285,130){\put(0,10){\circle*{4}} \put(0,40){\circle*{4}}
\put(25,0){\circle*{4}} \put(25,50){\circle*{4}}
\put(40,25){\circle*{4}} \put(55,0){\circle*{4}}
\put(55,50){\circle*{4}} \put(0,10){\line(5,-2){25}}
\put(25,0){\line(3,5){15}} \put(40,25){\line(-3,5){15}}
\put(25,50){\line(-5,-2){25}} \put(0,40){\line(0,-1){30}}
\put(55,0){\line(-1,0){30}} \put(55,0){\line(-3,5){15}}
\put(55,50){\line(-1,0){30}} \put(55,50){\line(-3,-5){15}} }

\put(360,130){\put(0,0){\circle*{4}} \put(40,0){\circle*{4}}
\put(20,15){\circle*{4}} \put(0,30){\circle*{4}}
\put(20,32.5){\circle*{4}} \put(40,30){\circle*{4}}
\put(20,50){\circle*{4}} \put(0,0){\line(1,0){40}}
\put(20,15){\line(-4,-3){20}} \put(20,15){\line(4,-3){20}}
\put(0,0){\line(0,1){30}} \put(20,15){\line(0,1){35}}
\put(40,0){\line(0,1){30}} \put(20,50){\line(-1,-1){20}}
\put(20,50){\line(1,-1){20}} }

\put(15,200){${\rm{EX}}(6,C_4)$, ${\rm{ex}}(6,C_4)=7$}

\put(20,220){\put(0,0){\circle*{4}} \put(0,30){\circle*{4}}
\put(20,15){\circle*{4}} \put(40,0){\circle*{4}}
\put(40,30){\circle*{4}} \put(20,35){\circle*{4}}
\put(0,0){\line(4,3){20}} \put(0,0){\line(0,1){30}}
\put(0,30){\line(4,-3){20}} \put(20,15){\line(4,3){20}}
\put(20,15){\line(4,-3){20}} \put(40,0){\line(0,1){30}}
\put(20,15){\line(0,1){20}} }

\put(80,220){\put(0,0){\circle*{4}} \put(0,30){\circle*{4}}
\put(20,15){\circle*{4}} \put(40,0){\circle*{4}}
\put(40,30){\circle*{4}} \put(60,0){\circle*{4}}
\put(0,0){\line(4,3){20}} \put(0,0){\line(0,1){30}}
\put(0,30){\line(4,-3){20}} \put(20,15){\line(4,3){20}}
\put(20,15){\line(4,-3){20}} \put(40,0){\line(0,1){30}}
\put(40,0){\line(1,0){20}} }

\put(160,220){\put(0,0){\circle*{4}} \put(0,30){\circle*{4}}
\put(20,15){\circle*{4}} \put(40,15){\circle*{4}}
\put(60,0){\circle*{4}} \put(60,30){\circle*{4}}
\put(0,0){\line(4,3){20}} \put(0,0){\line(0,1){30}}
\put(0,30){\line(4,-3){20}} \put(20,15){\line(1,0){20}}
\put(40,15){\line(4,3){20}} \put(40,15){\line(4,-3){20}}
\put(60,0){\line(0,1){30}} }

\put(240,210){\put(20,10){\circle*{4}} \put(20,40){\circle*{4}}
\put(45,0){\circle*{4}} \put(45,50){\circle*{4}}
\put(60,25){\circle*{4}} \put(0,25){\circle*{4}}
\put(20,10){\line(5,-2){25}} \put(45,0){\line(3,5){15}}
\put(60,25){\line(-3,5){15}} \put(45,50){\line(-5,-2){25}}
\put(20,40){\line(0,-1){30}} \put(0,25){\line(4,-3){20}}
\put(0,25){\line(4,3){20}} }

\put(15,295){${\rm{EX}}(1,C_4)$,} \put(15,280){${\rm{ex}}(1,C_4)=0$}
\put(85,270){\line(0,1){85}}

\put(42.5,330){\put(0,0){\circle*{4}} }

\put(100,295){${\rm{EX}}(2,C_4)$,}
\put(100,280){${\rm{ex}}(2,C_4)=1$} \put(170,270){\line(0,1){85}}

\put(112.5,330){\put(0,0){\circle*{4}} \put(30,0){\circle*{4}}
\put(0,0){\line(1,0){30}} }

\put(185,295){${\rm{EX}}(3,C_4)$,}
\put(185,280){${\rm{ex}}(3,C_4)=3$} \put(255,270){\line(0,1){85}}

\put(197.5,320){\put(0,0){\circle*{4}} \put(30,0){\circle*{4}}
\put(15,20){\circle*{4}} \put(0,0){\line(3,4){15}}
\put(0,0){\line(1,0){30}} \put(30,0){\line(-3,4){15}} }

\put(270,295){${\rm{EX}}(4,C_4)$,}
\put(270,280){${\rm{ex}}(4,C_4)=4$} \put(345,270){\line(0,1){85}}

\put(280,315){\put(0,0){\circle*{4}} \put(0,30){\circle*{4}}
\put(20,15){\circle*{4}} \put(40,15){\circle*{4}}
\put(0,0){\line(4,3){20}} \put(0,0){\line(0,1){30}}
\put(0,30){\line(4,-3){20}} \put(20,15){\line(1,0){20}} }

\put(360,295){${\rm{EX}}(5,C_4)$,}
\put(360,280){${\rm{ex}}(5,C_4)=6$}

\put(370,315){\put(0,0){\circle*{4}} \put(0,30){\circle*{4}}
\put(20,15){\circle*{4}} \put(40,0){\circle*{4}}
\put(40,30){\circle*{4}} \put(0,0){\line(4,3){20}}
\put(0,0){\line(0,1){30}} \put(0,30){\line(4,-3){20}}
\put(20,15){\line(4,3){20}} \put(20,15){\line(4,-3){20}}
\put(40,0){\line(0,1){30}} }
\end{picture}

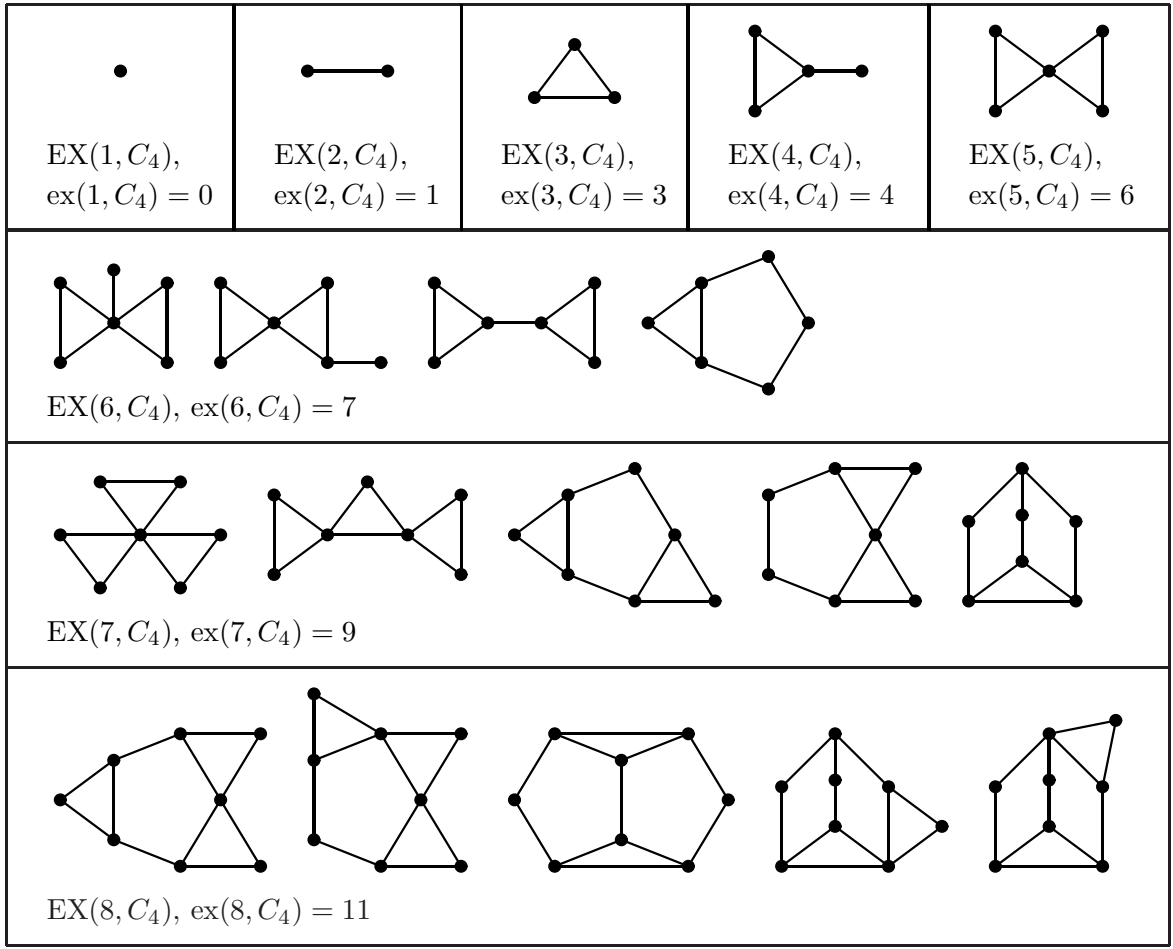
\small Figure~4. Extremal graphs for $C_4$ with $n\in [1,8]$.
\end{center}

For $n\geq 9$, we first define a class of
$\overrightarrow{P_{1,3}}$-free digraphs on $n$ vertices, namely
\textit{symmetric multipartite digraph}, denoted by
$\overrightarrow{S_n}$, which satisfies the following properties
(see Figure~5):
\begin{compactenum}[\rm (i)]
  \item $V(\overrightarrow{S_n})$ has an almost balanced partition $\{X, Y, Z\}$;\label{P13:a}
  \item $X$, $Y$ and $Z$ are independent sets of $\overrightarrow{S_n}$;
  \item $X$, $Y$, $Z$, respectively, has a partition $\{X_1, X_2\}$, $\{Y_1,
  Y_2\}$, $\{Z_1, Z_2\}$ (possibly $X_i, Y_j$ or $Z_k=\emptyset$);
  \item $X_1\mapsto Y_1\cup Z_2$, $Y_1\mapsto Z_1\cup X_2$, $Z_1\mapsto X_1\cup Y_2$,
  $X_2\mapsto Y_2\cup Z_1$, $Y_2\mapsto Z_2\cup X_1$, $Z_2\mapsto X_2\cup
  Y_1$. \label{P13:b}
\end{compactenum}
\noindent{\bf Remark:} There is no relationship between the size of
$X_1$ and $X_2$, and the same for $Y_1$ and $Y_2$, $Z_1$ and $Z_2$.
Note that there may be more than one non-isomorphic
$\overrightarrow{S_n}$ for given $n$. However, all
$\overrightarrow{S_n}$ have the same arc number, and the underlying
multigraph of $\overrightarrow{S_n}$ is only the Tur\'an graph
$T_{n,3}$. We denote by $\overrightarrow{\mathcal{S}_n}$ the set of
all $\overrightarrow{S_n}$.

\begin{center}
\begin{picture}(180,160)
\thicklines

\put(100,135){\line(2,-1){40}} \put(96,131){\line(2,-3){48}}
\put(150,99){\line(0,-1){38}} \put(144,101){\line(-2,-3){48}}
\put(140,45){\line(-2,-1){40}} \put(139,50){\line(-1,0){98}}
\put(80,25){\line(-2,1){40}} \put(84,29){\line(-2,3){48}}
\put(30,61){\line(0,1){38}} \put(36,59){\line(2,3){48}}
\put(40,115){\line(2,1){40}} \put(41,110){\line(1,0){98}}

\put(120,125){\vector(2,-1){5}} \put(120,95){\vector(2,-3){5}}
\put(150,80){\vector(0,-1){5}} \put(120,65){\vector(-2,-3){5}}
\put(120,35){\vector(-2,-1){5}} \put(90,50){\vector(-1,0){5}}
\put(60,35){\vector(-2,1){5}} \put(60,65){\vector(-2,3){5}}
\put(30,80){\vector(0,1){5}} \put(60,95){\vector(2,3){5}}
\put(60,125){\vector(2,1){5}} \put(90,110){\vector(1,0){5}}

\put(90,140){\circle{24}} \put(150,110){\circle{24}}
\put(150,50){\circle{24}} \put(90,20){\circle{24}}
\put(30,50){\circle{24}} \put(30,110){\circle{24}}

\put(84,136){$X_1$} \put(144,106){$Z_2$} \put(144,46){$Y_1$}
\put(84,16){$X_2$} \put(24,46){$Z_1$} \put(24,106){$Y_2$}
\end{picture}

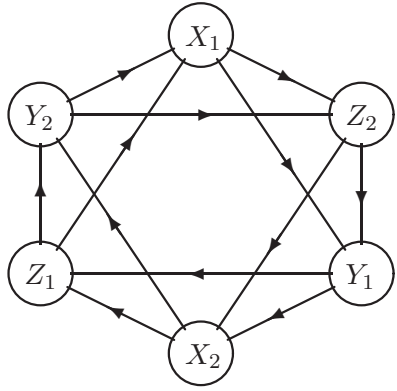
\small Figure~5. Digraph $\overrightarrow{S_n}$.
\end{center}

\begin{theorem}\label{p13}
For all $n\ge 9$,
\[
{\rm
ex}(n,\overrightarrow{P_{1,3}})=a(\overrightarrow{S_n})=\left\lfloor\frac{n^2}{3}\right\rfloor.
\]
%$${\rm ex}(n,\overrightarrow{P_{1,3}})={\rm ex}(n,K_4)=\left\lfloor\frac{n^2}{3}\right\rfloor=\left\{\begin{array}{ll}
%n^2/3,      & \mbox{if } n\equiv 0 \bmod 3;\\
%(n^2-1)/3,  & \mbox{if } n\equiv 1,2 \bmod 3.
%\end{array} \right.$$
Furthermore,
${\rm{EX}}(n,\overrightarrow{P_{1,3}})=\overrightarrow{\mathcal{S}_n}$.
\end{theorem}

For $\overrightarrow{K_{2,2}}$, we are surprised to find that to
calculate the Tur\'an number of $\overrightarrow{K_{2,2}}$ is as
difficult as to $C_4$ in bipartite graphs. Given a bipartite graph
$H$, let ${\rm ex}(n,n;H)$ denote the maximum number of edges in a
$H$-free bipartite graph with partition sets of equal size $n$; and
let ${\rm EX}(n,n;H)$ denote the set of all the $H$-free bipartite
graph with ${\rm ex}(n,n;H)$ edges and with partition sets of equal
size $n$.

Given a digraph $D$ with $V(D)=\{v_i: i\in[1,n]\}$, we define the
\emph{bipartite representation} of $D$ as a bipartite graph $B$ with
partition sets $X=\{x_i: i\in[1,n]\}$ and $Y=\{y_i: i\in[1,n]\}$
such that $x_iy_j\in E(B)$ if and only if $v_iv_j\in A(D)$. We
denote by ${\rm Brp}(D)$ the bipartite representation of $D$.

\begin{theorem}\label{anticycle}
For $n\in \mathbb{N}^*$ with $n\ge 3$,
\[
{\rm ex}(n,\overrightarrow{K_{2,2}})={\rm ex}(n,n;C_4).
\]
{Furthermore, ${\rm{EX}}(n,\overrightarrow{K_{2,2}})=\{D: {\rm
Brp}(D)\in {\rm{EX}}(n,n; C_4)\}$.}
\end{theorem}

The paper is organized as follows. The results of directed paths are
proved in Section~\ref{Sec-path}. The results of directed cycles are
proved in Section~\ref{Sec-cycle}. In Section~\ref{Sec-4-cycle}, we
prove the results of $\overrightarrow{P_{1,3}}$ and
$\overrightarrow{K_{2,2}}$. Finally, we close the paper with some
remarks and conjectures in Section~\ref{Sec-re.}.

\section{The proofs for directed paths}\label{Sec-path}

In this section, we prove Theorem \ref{3-path} and Theorem
\ref{P_k}. Corollary~\ref{P_k-upper-bound} can be directly obtained
from the proof of Theorem~\ref{P_k} (see bellow). From now on, let
$n,k\in \mathbb{N}^*$ with $k\geq 2$, $n=qk+r,\ 0\leq r<k$.

\begin{proof}[{\bf Proof of Theorem~\ref{3-path}}]
Let $D\in {\rm EX}(n,\overrightarrow{P_{3}})$, and let $\mathcal{C}$
denote the set of all $\overrightarrow{C_2}$ in $D$,
$t=|\mathcal{C}|$. Since $D$ is $\overrightarrow{P_3}$-free, every
$\overrightarrow{C_2}$ is a component of $D$. Let $D'$ be the
digraph obtained from $D$ by removing all the
$\overrightarrow{C_2}$-components. Then ${\rm Umg}(D')$ is a simple
graph. Moreover, ${\rm Umg}(D')$ is also $K_{3}$-free by
Theorem~\ref{H-path}. Hence $a(D')\leq {\rm
ex}(|V(D')|,K_3)=\lfloor{\frac{(n-2t)^2}{4}}\rfloor$ by
Theorem~\ref{K_k-free}. Therefore, for $n\geq 3$, we obtain that
\[
a(D)\leq 2t+\left\lfloor{\frac{(n-2t)^2}{4}}\right\rfloor\leq
\left\lfloor{\frac{n^2}{4}}\right\rfloor.
\]
Note that the first equality holds only if $a(D')={\rm
ex}(|V(D')|,K_3)$, and the second equality holds only if $t=0$ for
$n\geq 3$, or $t=\lfloor\frac{n}{2}\rfloor$ for $n=3,4$.
Furthermore, when $t=0$, then ${\rm Umg}(D)$ is the $K_{3}$-free
extremal graph $T_{n,2}$. It is easy to see that an orientation of
$T_{n,2}$ is $\overrightarrow{P_3}$-free if and only if it is a
transitive orientation. If $n\in \{3,4\}$ and
$t=\lfloor\frac{n}{2}\rfloor$, then clearly
$D=\overleftrightarrow{\varGamma_{n,2}}$. Thus we have ${\rm
EX}(n,\overrightarrow{P_3})=\overrightarrow{\mathcal{T}_{n,2}} \cup
\{\overleftrightarrow{\varGamma_{n,2}} \}$ for $n\in \{3,4\}$, and
${\rm
EX}(n,\overrightarrow{P_3})=\overrightarrow{\mathcal{T}_{n,2}}$ for
$n\geq 5$.
\end{proof}

\begin{proof}[{\bf Proof of Theorem~\ref{P_k}}]
From now on, let $D\in {\rm EX}(n,\overrightarrow{P_{k+1}})$. Since
$\overrightarrow{T_{n,k}}$ is a $\overrightarrow{P_{k+1}}$-free
digraph, $a(D)\ge a(\overrightarrow{T_{n,k}})={\rm ex}(n,K_{k+1})$.
Our main aim is to prove $D\in \overrightarrow{\mathcal{T}_{n,k}}$
for large $n$. For this, we will need two key claims.

\begin{claim}\label{inducedsub}
The digraph $D$ contains an induced subdigraph
$\overrightarrow{T_{k^2,k}}$.
\end{claim}

\begin{proof}
Let $\mathcal{C}$ still denote the set of all $\overrightarrow{C_2}$
in $D$ and $t=|\mathcal{C}|$. By Theorem~\ref{E-G-path}, we have
$t\leq {\rm ex}(n,P_{k+1})$. Consider the graph $G:={\rm Ung}(D)$.
On the one hand, since Theorem~\ref{H-path} implies that $G$ is
$K_{k+1}$-free, we have $e(G)\leq {\rm ex}(n, K_{k+1})$. (We remark
that this implies that $a(D)\leq{\rm ex}(n,P_{k+1})+{\rm
ex}(n,K_{k+1})$, and thus Corollary~\ref{P_k-upper-bound} holds.) On
the other hand, since $a(D)\geq e(T_{n,k})$, by Theorem
\ref{E-G-path}, we have
\[
e(G)=a(D)-t\geq e(T_{n,k})-{\rm ex}(n,P_{k+1})\geq
e(T_{n,k-1})+\frac{1}{2k^2}n^2,
\]
for all $n\geq k^4$. By Theorem~\ref{K^s_k-free}, there is an
integer $n_0$ (which depends only on $k$) such that for all $n\geq
n_0$, $G$ contains a $K_k^k=T_{k^2,k}$ as a subgraph.
Recall that $G$ is $K_{k+1}$-free. Therefore, $G$ contains an induced subgraph $H$ isomorphic to $T_{k^2,k}$.

Now consider the orientation of $H$ in $D$. Let
$\{U_1,U_2,\ldots,U_k\}$ be the partition of $V(H)$,
$S=\{u_1,u_2,\ldots,u_k\}$ be a subset of $V(H)$ with $u_i\in U_i$,
and $D[S]$ be the induced subgraph of $D$ on $S$. By Theorem
\ref{H-path}, $D[S]$ contains a $\overrightarrow{P_k}$. Without loss
of generality, we assume that $u_i\rightarrow u_{i+1}$ for
$i\in[1,k-1]$.

Given $v_j\in U_j\setminus\{u_j\}$, we claim that $v_j\nrightarrow
u_i$ in $D$ for $i\in[1,j-1]$. Suppose not, then there is an arc
$v_j\rightarrow u_i$ with $i$ as small as possible. Clearly, $i\neq
1$, otherwise $\overrightarrow{P}:=v_ju_1\ldots u_k$ is a
$\overrightarrow{P_{k+1}}$ in $D$, a contradiction. Now we assume
that $i\in[2,j-1]$. Note that we must have $u_{i-1}\rightarrow v_j$
in $D$ from the minimality of $i$. In this case,
$\overrightarrow{P}:=u_1\ldots u_{i-1}v_ju_i\ldots u_k$ is also a
$\overrightarrow{P_{k+1}}$, a contradiction. Thus, we have
$u_i\mapsto v_j$ in $D$ for all $v_j\in U_j\setminus\{u_j\}$ and
$1\leq i<j\le k$. By a similar analysis, for $v_i\in
U_i\setminus\{u_i\}$, we can show that $v_i\mapsto u_j$ in $D$ for
$1\le i<j\leq k$.

Next we show that for every $v_i\in U_i$ and $v_j\in U_j$ with
$1\leq i<j\leq k$, $v_j\nrightarrow v_i$ holds. If $v_i=u_i$ and
$v_j\neq u_j$, or $v_i\neq u_i$ and $v_j=u_j$, then we are done by
the analysis above. Suppose now that $v_i\neq u_i$ and $v_j\neq
u_j$. Then we have $u_i\rightarrow v_j$ and $v_i\rightarrow
u_{i+1}$. If $v_j\rightarrow v_i$, then
$\overrightarrow{P}:=u_1\ldots u_iv_jv_iu_{i+1}\ldots u_k$ is a
$\overrightarrow{P_{k+2}}$ in $D$, a contradiction. Finally, we
consider the case $v_i=u_i$ and $v_j=u_j$. Let $v'_i$ be a vertex in
$U_i\setminus\{u_i\}$. We have $u_{i-1}\rightarrow v'_i$ (if $i\geq
2$) and $v'_i\rightarrow u_{i+1}$. It follows that
$\overrightarrow{P}:=u_1\ldots u_{i-1}v'_iu_{i+1}\ldots u_k$ is also
a $\overrightarrow{P_k}$ in $D$, and $v_j\nrightarrow v_i$ as well.
Thus we have $U_i\mapsto U_j$ in $D$ for $1\leq i<j\leq k$. This
implies that the subdigraph of $D$ induced by $V(H)$ is a
$\overrightarrow{T_{k^2,k}}$.
\end{proof}

According to Claim \ref{inducedsub}, let $D_1$ be an induced
$\overrightarrow{T_{k^2,k}}$ of $D$. For simplicity, let
$\{U_1,U_2,\ldots,U_k\}$ be the partition of $V(D_1)$ with
$U_i\mapsto U_j$ for $1\leq i<j\leq k$. Thus $|U_i|=k$ for all
$i\in[1,k]$. Note that $U_k \nrightarrow v$ for every $v\in V(D)$
since $D$ is $\overrightarrow{P_{k+1}}$-free. Thus for every $v\in
V(D)$, there is a smallest $i$ such that $U_i \nrightarrow v$. We
define a partition $\{V_1,V_2,\dots, V_k\}$ of $V(D)$ as follow:
\[
V_j:= \left \{v\in V(D): j=\min \left \{i\in [1,k]: U_i\nrightarrow v
\right \}\right \}.
\]
Clearly, $V(D)=\bigcup _{i=1}^{k} V_i$, $V_i\cap V_j=\emptyset$ for
$1\leq i<j\leq k$, and $U_i\subseteq V_i$ for all $i\in [1,k]$. For
every $v\in V(D)$, by the arcs-maximality of $D$ (i.e., $D\in {\rm
EX}(n,\overrightarrow{P_{k+1}})$), we will obtain the following
result.

\begin{claim}\label{vretex-pro}
For $s\in [1,k]$, if $v\in V_s$ then $\bigcup_{i=s}^kV_i\nrightarrow v$ and
$v\rightarrow \bigcup_{i=s+1}^kV_i$.
\end{claim}

\begin{proof}
Suppose the claim is not true for some $s$. We choose such an $s$
that is as small as possible, and let $v\in V_s$ not satisfying the
assertion. We first show that there is a directed path of order
$k-s+1$ starting from $v$. Let $u\in U_s\setminus\{v\}$ and $D'$ be
the digraph obtained from $D$ by adding the arc $uv$. Since $D$ is
extremal for $\overrightarrow{P_{k+1}}$, $D'$ contains a directed
path $\overrightarrow{P}$ of order $k+1$ which passes through the
added arc $uv$. Let $x,y$ be the origin and terminus of
$\overrightarrow{P}$. If $D$ has no path of order $k-s+1$ starting
from $v$, then $|V(\overrightarrow{P}[v,y])|\leq k-s$ and
$|V(\overrightarrow{P}[x,u])|\geq s+1$. Thus we can take $u_i\in
U_i\setminus V(\overrightarrow{P})$ with $i\in[s+1,k]$, then
$\overrightarrow{P}[x,u]uu_{s+1}\ldots u_k$ is a
$\overrightarrow{P_{k+1}}$ in $D$, a contradiction. For simplicity,
let $\overrightarrow{Q}$ be a directed path of order $k-s+1$
starting from $v$.

Let $W=\{w\in\bigcup_{i=s}^kV_i: w\rightarrow v\}$. We will show
that $W=\emptyset$. First consider $s=1$. For $v\in V_1$, if
$|W|\geq k$, then there is a vertex $w\in W\setminus
V(\overrightarrow{Q})$, which implies $wv\overrightarrow{Q}$ is a
$\overrightarrow{P_{k+1}}$, a contradiction. Now assume that
$1\leq|W|\leq k-1$. Let $w\in W$ be arbitrary. If $v\rightarrow u_2$
for some $u_2\in U_2\backslash\{w\}$, then we can take $u_i\in U_i$
with $i\in[3,k]$, and $wvu_2\ldots u_k$ is a
$\overrightarrow{P_{k+1}}$ in $D$, a contradiction. Thus we have
$v\nrightarrow U_2\backslash\{w\}$. Specially, if $|W|\geq 2$, then
$v\nrightarrow U_2$. Let $D'$ be the digraph on $V(D)$ with arc set
\[
A(D')=(A(D)\setminus\{wv: w\in W\})\cup\{vu: u\in U_2\},
\]
which yields another $\overrightarrow{P_{k+1}}$-free digraph with more arcs than $D$, contradicting the choice of $D$.
%For every $u\in\bigcup^k_{i=2}V_i$, since $W= \emptyset$ and there exists $u_1\in U_1$ such that $u_1\rightarrow u$, there must be $v\rightarrow u$, by the arc-maximality of $D$.
%Therefore, $W= \emptyset$, i.e., $\bigcup^k_{i=1}V_i \nrightarrow v$ in $D$.

Next, let $s\in [2,k-1]$. By the choice of $v$, we have that
$D[\bigcup_{i=1}^{s-1}V_i]$ is $\overrightarrow{P_s}$-free and
$\bigcup_{i=1}^{s-1}V_i\mapsto\bigcup_{i=s}^kV_i$. If $|W|\geq
k-s+1$, then there is a vertex $w\in W\setminus
V(\overrightarrow{Q})$. Thus we can take $u_i\in U_i$ with
$i\in[1,s-1]$, then $u_1\ldots u_{s-1}wv\overrightarrow{Q}$ is a
$\overrightarrow{P_{k+1}}$ in $D$, a contradiction. Now assume that
$1\leq|W|\leq k-s$. Let $w\in W$ be arbitrary. If $v\rightarrow
u_{s+1}$ for some $u_{s+1}\in U_{s+1}\backslash\{w\}$, then we can
take $u_i\in U_i$ with $i\in[1,s-1]\cup[s+2,k]$, and $u_1\ldots
u_{s-1}wvu_{s+1}\ldots u_k$ is a $\overrightarrow{P_{k+1}}$ in $D$,
a contradiction. Thus we have $v\nrightarrow
U_{s+1}\backslash\{w\}$. Specially, if $|W|\geq 2$, then
$v\nrightarrow U_{s+1}$. Let $D'$ be the digraph on $V(D)$ with arc
set
\[
A(D')=(A(D)\setminus\{wv: w\in W\})\cup\{vu: u\in U_{s+1}\}.
\]
Clearly $|A(D')|>|A(D)|$, implying that $D'$ contains a directed
path $\overrightarrow{P}=\overrightarrow{P_{k+1}}$ which passes
through an arc $vu$ with $u\in U_{s+1}$. Let $x,y$ be the origin and
terminus of $\overrightarrow{P}$. Clearly
$V(\overrightarrow{P}[x,v])\setminus\{v\}\subseteq\bigcup_{i=1}^{s-1}V_i$
and thus $|V(\overrightarrow{P}[x,v])|\leq s$. It follows that
$|V(\overrightarrow{P}[u,y])|\geq k-s+1$. Thus we can take $u_i\in
U_i$ with $i\in[1,s]$, then $u_1\ldots u_su\overrightarrow{P}[u,y]$
is a $\overrightarrow{P_{k+1}}$ of $D$, a contradiction.

Finally let $s=k$. Thus $W\subseteq V_k$. If $W\neq\emptyset$, then
we can take $w\in W$, $u_i\in U_i$ for $i\in[1,k-1]$, and $u_1\ldots
u_{k-1}wv$ is a $\overrightarrow{P_{k+1}}$ of $D$, a contradiction.
In any case, we proved that $W=\emptyset$, i.e.,
$\bigcup_{i=s}^kV_i\nrightarrow v$.

Suppose now that there is a vertex $u\in\bigcup_{i=s+1}^kV_i$ such
that $v\nrightarrow u$ in $D$. Let $D'$ be the digraph obtained from
$D$ by adding the arc $vu$. Then $D'$ contains a path
$\overrightarrow{P}=\overrightarrow{P_{k+1}}$ which passes through
the arc $vu$. Let $x,y$ be the origin and terminus of
$\overrightarrow{P}$. Recall that $\bigcup_{i=s}^kV_i\nrightarrow v$
and $D[\bigcup_{i=1}^{s-1}V_i]$ is $\overrightarrow{P_s}$-free. We
have that $|V(\overrightarrow{P}[x,v])|\leq s$ and
$|V(\overrightarrow{P}[u,y])|\geq k-s+1$ (we remark that if $s=1$,
then $x=v$, and $|V(\overrightarrow{P}[u,y])|\geq k$). Note that
$U_s\nrightarrow u$ does not hold, implying that there is a vertex
$u_s\in U_s$ with $u_s\rightarrow u$. Thus we can take $u_i\in U_i$,
$i\in[1,s-1]$, and $u_1\ldots u_su\overrightarrow{P}[u,y]$ contains
a $\overrightarrow{P_{k+1}}$ of $D$, a contradiction. Thus we have
$v\rightarrow u$ for all $u\in\bigcup_{i=s+1}^kV_i$.
\end{proof}

According to Claim \ref{vretex-pro}, it is easy to see that the
partition $\{V_1,V_2,\dots, V_k\}$ of $V(D)$ satisfies the
Properties~(\ref{path-item:b}) and~(\ref{path-item:c}) in the
definition of $\overrightarrow{T_{n,k}}$. When $D$ also satisfies
the Property~(\ref{path-item:a}) in the definition of
$\overrightarrow{T_{n,k}}$, $D$ has the maximum number of arcs.
Therefore, $D\in \overrightarrow{\mathcal{T}_{n,k}}$.
\end{proof}

\section{The proof for directed cycles}\label{Sec-cycle}
In this section, we aim to prove Theorem~\ref{C_k}. For simplicity,
let $f(n,k)=a(\overrightarrow{F_{n,k}})$ for all $n,k\in
\mathbb{N}^*$.
%We restate the Theorem for convenience.
%\medskip
%\noindent\textbf{Theorem 1.8.}
%\emph{Let $k,n\in \mathbb{N}^*$, $n=qk+r$, $0\leq r<k$. Then
%$$\dex(n,\overrightarrow{C_{k+1}})=a(\overrightarrow{F_{n,k}})=\frac{1}{2}n^2+\frac{k-2}{2}n-\frac{r(k-r)}{2}.$$
%Furthermore, ${\rm
%EX}(n,\overrightarrow{C_{k+1}})=\overrightarrow{\mathcal{F}_{n,k}}$ for $k\geq 3$. }

\begin{proof}[{\bf Proof of Theorem~\ref{C_k}}]
%Let us now show that, $\ria{ex}(n,\overrightarrow{C_k})= f(n,k)$ holds for all $n>1$ and $k>1$.
We use double induction on $n$ and $k$. For $k=1,2$, we are done by
Theorem \ref{Brown-Harary}. Now we assume that $k\geq 3$. When $n\le
k$, clearly $\overleftrightarrow{K_n}=\overrightarrow{F_{n,k}}$ is
the unique extremal digraph for $\overrightarrow{C_{k+1}}$ and ${\rm
ex}(n,\overrightarrow{C_{k+1}})=f(n,k)$. For
the induction step, assume that $n\ge k+1$ and ${\rm
ex}(n',\overrightarrow{C_{k+1}})= f(n',k)$ holds for all $n'<n$.

Let $D\in {\rm EX}(n,\overrightarrow{C_{k+1}})$, then $a(D)\ge f(n,k)$
since $\overrightarrow{F_{n,k}}$ is $\overrightarrow{C_{k+1}}$-free.
Due to $f(n,k)>f(n,k-1)$ for all $n\geq k$, $D$ contains a directed
cycle $\overrightarrow{C}=\overrightarrow{C_k}$. Let
$V(\overrightarrow{C})=\{u_1,u_2,\ldots,u_k\}$ with $u_iu_{i+1}\in
A(\overrightarrow{C})$ (the subscripts are taken modulo $k$) and $V'=V(D)\setminus V(\overrightarrow{C})$.
Clearly $a(D[V(\overrightarrow{C})])\leq k(k-1)$. By the induction
hypothesis, $D[V']$ has at most $f(n-k,k)$ arcs.

For $v\in V'$, we claim that there
are at most $k$ arcs between $V(\overrightarrow{C})$ and $v$.
%(with forward or opposite direction)
For otherwise there will be two
vertices $u_i,u_{i+1}$ with $u_i\rightarrow v$ and $v\rightarrow
u_{i+1}$, and thus
$\overrightarrow{C'}=u_ivu_{i+1}\overrightarrow{C}[u_{i+1},u_i]$ is
a $\overrightarrow{C_{k+1}}$ in $D$. Hence, we have
\[
a(D)\leq k(k-1)+k(n-k)+f(n-k,k)=f(n,k).
\]
The equality on the right follows by the definition of $f(n,k)$.
Since $a(D)\geq f(n,k)$, we have equality in above formula. Thus,
$D[V(\overrightarrow{C})]$ must be a $\overleftrightarrow{K_k}$, the number of arcs
between $V(\overrightarrow{C})$ and every vertex in $V'$ is exactly $k$, and $D[V']$
is an extremal digraph for $\overrightarrow{C_{k+1}}$ and $n-k$.

For every $v\in V'$, since
$D[V(\overrightarrow{C})]$ is a $\overleftrightarrow{K_k}$, these
$k$ arcs between $v$ and $V(\overrightarrow{C})$ must have the same
direction. For two distinct vertices $u,v\in V'$, if $uv$ and $vu$ are all in $A(D)$, then the
direction of the arcs between $V(\overrightarrow{C})$ and $u$ must be
the same as that between $V(\overrightarrow{C})$ and $v$. Therefore, by
recursive induction, we can obtain that the structure of $D$ can
only be a $\overrightarrow{F_{n,k}}$.
%In particular, when $k=3$, $D[V(C)]$ is a $\overrightarrow{C_2}$ in $D$. Set $V(C)=\{x_1,x_2\}$. Fix $v\in V(D)\setminus \{x_1,x_2\}$, two arcs between $v$ and $\{x_1,x_2\}$ can have four forms in $D$, namely, $v\mapsto V(C)$, $V(C)\mapsto v$ or $v\leftrightarrow x_i$ for $i\in [2]$. We define a partition $\{V_1, V_2, V_3, V_4\}$ of $V(D)\setminus V(C)$ as follow: $V_i:=\{v\in V(D)\setminus V(C) \mid v\leftrightarrow x_i\}$, $V_3:=\{v\in V(D)\setminus V(C) \mid v\mapsto V(C)\}$ and $V_4:=\{v\in V(D)\setminus V(C) \mid V(C)\mapsto v\}$. Since $D$ is $\overrightarrow{C_3}$-free, we can obtain that $V_1$ and $V_2$ are independent sets of $D$. By the arcs-maximality of $D$, it has $V_1\leftrightarrow V_2$, $V_3\mapsto V_1$, $V_1\mapsto V_4$, $V_3\mapsto V_2$, $V_3\mapsto V_2$ and $V_3\mapsto V_4$. For the vertices in $V_3$ or $V_4$, we continue the induction, then we can find that $D$ is exactly the form described in Theorem~\ref{Brown-Harary}.
\end{proof}

\section{The proofs for orientations of $C_4$}\label{Sec-4-cycle}

In this section, we will prove {Theorem~\ref{p13small},
Theorem~\ref{p13} and Theorem~\ref{anticycle}, respectively.} To
prepare for this we will need some notation. We
follow~\cite{bang2008digraphs} and~\cite{MR3822066}. Let $G=(V,E)$
be a graph, for $v\in V$, the \textit{neighborhood} of $v$ is
denoted by $N_G(v):=\{u\in V: uv \in E\}$, and the degree of $v$ is
$d_G(v)=|N_G(v)|$.
%More generally for $U\subseteq V$, the neighbors
%in $V\setminus U$ of vertices in $U$ are called \textit{neighbors}
%of $U$; their set is denoted by $N_G(U):=\bigcup_{u\in
%U}N_G(u)\setminus U$.
Analogously, given a digraph $D=(V,A)$ and
$v\in V$, we use the following notation:
\[
N^+_D(v):=\{u\in V: (v,u)\in A\},\ N^-_D(v):=\{w\in V: (w,v)\in A\}.
\]
The set $N^+_D(v)$, $N^-_D(v)$ and $N_D(v)=N^+_D(v)\cup N^-_D(v)$
are called the \textit{out-neighborhood}, \textit{in-neighborhood}
and \textit{neighborhood} of $v$. The \textit{out-degree}, the
\textit{in-degree}, and the \textit{degree} of $v$ are denoted by
$d^+_D(v):=|N^+_D(v)|$, $d^-_D(v):=|N^-_D(v)|$, and
$d_D(v):=d^+_D(v)+d^-_D(v)$. For $i,j,k \in \mathbb{N}^*$, by
$\varTheta_{i,j,k}$ we mean the graph (or multigraph) consisting of
three internally-disjoint paths of lengths $i,j,k$, respectively,
with the same origin and the same terminus.

Next, we state several lemmas that are useful for the proofs of
Theorem~\ref{p13small} and Theorem~\ref{p13}.

\begin{lemma}\label{LeOrientation}
Every orientation of $K_4$ and $\varTheta_{1,1,3}$ contains a
$\overrightarrow{P_{1,3}}$.
\end{lemma}
In Figure~6 we list all the (non-isomorphic) orientations of $K_4$
and $\varTheta_{1,1,3}$ together with the
$\overrightarrow{P_{1,3}}$-subdigraphs, which directly proves
Lemma~\ref{LeOrientation}.

\begin{lemma}\label{LeSubgraph}
For $n\leq 8$, every non-empty $K_4$-free graph $G$ on $n$ vertices
contains a $C_4$-free subgraph $H$ with $e(H)>e(G)/2$.
\end{lemma}
The proof of Lemma~\ref{LeSubgraph} needs some technic analysis,
which we postpone to the last part of the section.
\begin{center}
\begin{picture}(300,145)\label{P13}

\put(20,85){\put(0,0){\circle*{4}} \put(50,0){\circle*{4}}
\put(0,50){\circle*{4}} \put(50,50){\circle*{4}}
\put(0,0){\line(1,0){50}} \put(0,0){\line(0,1){50}}
\put(0,0){\line(1,1){50}} \put(0,50){\line(1,0){50}}
\put(0,50){\line(1,-1){50}} \put(50,0){\line(0,1){50}}
\put(25,50){\vector(1,0){2}} \put(15,15){\vector(1,1){2}}
\put(25,0){\vector(1,0){2}} \put(0,25){\vector(0,1){2}}
\put(15,35){\vector(1,-1){2}} \put(50,25){\vector(0,-1){2}}
\thicklines \textcolor{blue}{\put(0,0){\line(1,0){50}}
\put(0,0){\line(0,1){50}} \put(0,50){\line(1,0){50}}
\put(50,0){\line(0,1){50}}}}

\put(90,85){\put(0,0){\circle*{4}} \put(50,0){\circle*{4}}
\put(0,50){\circle*{4}} \put(50,50){\circle*{4}}
\put(0,0){\line(1,0){50}} \put(0,0){\line(0,1){50}}
\put(0,0){\line(1,1){50}} \put(0,50){\line(1,0){50}}
\put(0,50){\line(1,-1){50}} \put(50,0){\line(0,1){50}}
\put(25,50){\vector(1,0){2}} \put(15,15){\vector(1,1){2}}
\put(25,0){\vector(1,0){2}} \put(0,25){\vector(0,1){2}}
\put(15,35){\vector(-1,1){2}} \put(50,25){\vector(0,-1){2}}
\thicklines \textcolor{blue}{\put(0,0){\line(0,1){50}}
\put(0,0){\line(1,1){50}} \put(0,50){\line(1,-1){50}}
\put(50,0){\line(0,1){50}}}}

\put(160,85){\put(0,0){\circle*{4}} \put(50,0){\circle*{4}}
\put(0,50){\circle*{4}} \put(50,50){\circle*{4}}
\put(0,0){\line(1,0){50}} \put(0,0){\line(0,1){50}}
\put(0,0){\line(1,1){50}} \put(0,50){\line(1,0){50}}
\put(0,50){\line(1,-1){50}} \put(50,0){\line(0,1){50}}
\put(25,50){\vector(1,0){2}} \put(15,15){\vector(-1,-1){2}}
\put(25,0){\vector(1,0){2}} \put(0,25){\vector(0,1){2}}
\put(15,35){\vector(1,-1){2}} \put(50,25){\vector(0,-1){2}}
\thicklines \textcolor{blue}{\put(0,0){\line(0,1){50}}
\put(0,0){\line(1,1){50}} \put(0,50){\line(1,-1){50}}
\put(50,0){\line(0,1){50}}}}

\put(230,85){\put(0,0){\circle*{4}} \put(50,0){\circle*{4}}
\put(0,50){\circle*{4}} \put(50,50){\circle*{4}}
\put(0,0){\line(1,0){50}} \put(0,0){\line(0,1){50}}
\put(0,0){\line(1,1){50}} \put(0,50){\line(1,0){50}}
\put(0,50){\line(1,-1){50}} \put(50,0){\line(0,1){50}}
\put(25,50){\vector(1,0){2}} \put(15,15){\vector(1,1){2}}
\put(25,0){\vector(-1,0){2}} \put(0,25){\vector(0,1){2}}
\put(15,35){\vector(1,-1){2}} \put(50,25){\vector(0,-1){2}}
\thicklines \textcolor{blue}{\put(0,0){\line(1,0){50}}
\put(0,0){\line(1,1){50}} \put(0,50){\line(1,0){50}}
\put(0,50){\line(1,-1){50}}}}

\put(20,10){\put(0,0){\circle*{4}} \put(50,0){\circle*{4}}
\put(0,50){\circle*{4}} \put(50,50){\circle*{4}}
\put(0,0){\line(1,0){50}} \put(0,0){\line(0,1){50}}
\put(50,0){\line(0,1){50}} \qbezier(0,50)(25,60)(50,50)
\qbezier(0,50)(25,40)(50,50) \put(25,45){\vector(1,0){2}}
\put(25,55){\vector(-1,0){2}} \put(25,0){\vector(1,0){2}}
\put(0,25){\vector(0,-1){2}} \put(50,25){\vector(0,-1){2}}
\thicklines \textcolor{blue}{\put(0,0){\line(1,0){50}}
\put(0,0){\line(0,1){50}} \put(50,0){\line(0,1){50}}
\qbezier(0,50)(25,60)(50,50)}}

\put(90,10){\put(0,0){\circle*{4}} \put(50,0){\circle*{4}}
\put(0,50){\circle*{4}} \put(50,50){\circle*{4}}
\put(0,0){\line(1,0){50}} \put(0,0){\line(0,1){50}}
\put(50,0){\line(0,1){50}} \qbezier(0,50)(25,60)(50,50)
\qbezier(0,50)(25,40)(50,50) \put(25,45){\vector(1,0){2}}
\put(25,55){\vector(-1,0){2}} \put(25,0){\vector(1,0){2}}
\put(0,25){\vector(0,1){2}} \put(50,25){\vector(0,1){2}} \thicklines
\textcolor{blue}{\put(0,0){\line(1,0){50}} \put(0,0){\line(0,1){50}}
\put(50,0){\line(0,1){50}} \qbezier(0,50)(25,60)(50,50)}}

\put(160,10){\put(0,0){\circle*{4}} \put(50,0){\circle*{4}}
\put(0,50){\circle*{4}} \put(50,50){\circle*{4}}
\put(0,0){\line(1,0){50}} \put(0,0){\line(0,1){50}}
\put(50,0){\line(0,1){50}} \qbezier(0,50)(25,60)(50,50)
\qbezier(0,50)(25,40)(50,50) \put(25,45){\vector(1,0){2}}
\put(25,55){\vector(-1,0){2}} \put(25,0){\vector(1,0){2}}
\put(0,25){\vector(0,1){2}} \put(50,25){\vector(0,-1){2}}
\thicklines \textcolor{blue}{\put(0,0){\line(1,0){50}}
\put(0,0){\line(0,1){50}} \put(50,0){\line(0,1){50}}
\qbezier(0,50)(25,40)(50,50)}}

\put(230,10){\put(0,0){\circle*{4}} \put(50,0){\circle*{4}}
\put(0,50){\circle*{4}} \put(50,50){\circle*{4}}
\put(0,0){\line(1,0){50}} \put(0,0){\line(0,1){50}}
\put(50,0){\line(0,1){50}} \qbezier(0,50)(25,60)(50,50)
\qbezier(0,50)(25,40)(50,50) \put(25,45){\vector(1,0){2}}
\put(25,55){\vector(-1,0){2}} \put(25,0){\vector(-1,0){2}}
\put(0,25){\vector(0,1){2}} \put(50,25){\vector(0,-1){2}}
\thicklines \textcolor{blue}{\put(0,0){\line(1,0){50}}
\put(0,0){\line(0,1){50}} \put(50,0){\line(0,1){50}}
\qbezier(0,50)(25,60)(50,50)}}
\end{picture}

\small Figure~6. Orientations of $K_4$ and $\varTheta_{1,1,3}$.
\end{center}

\begin{proof}[{\bf Proof of Theorem \ref{p13small}.}]
Let $D$ be a $\overrightarrow{P_{1,3}}$-free extremal digraph on
$n\leq 8$ vertices. Then $a(D)\geq 2{\rm ex}(n,C_4)$. Let $G={\rm
Ung}(D)$, $G_1$ be the graph obtained from ${\rm Umg}(D)$ by
removing all edges of $C_2$, and $G_2$ be the subgraph of $G$ with
edge set $E(G)\setminus E(G_1)$. By Lemma~\ref{LeOrientation}, $G_1$
is $K_4$-free and every edge in $E(G_2)$ is not contained in a $C_4$
of $G$. Assume first that $E(G_1)\neq \emptyset$. By
Lemma~\ref{LeSubgraph}, $G_1$ has a $C_4$-free subgraph $H_1$ with
$e(H_1)>e(G_1)/2$. Now the union of $H_1$ and $G_2$ is a $C_4$-free
graph $H$ with $e(H)>a(D)/2$. It follows that
$\overleftrightarrow{H}$ is a $\overrightarrow{P_{1,3}}$-free
digraph with arc number greater than $D$, a contradiction.
Therefore, we conclude that $E(G_1)=\emptyset$ and $G=G_2$ is
$C_4$-free. Since $D$ is a extremal digraph, $e(G)=a(D)/2={\rm
ex}(n,C_4)$, which implies that $G\in {\rm EX}(n,C_4)$, and we are
done.
\end{proof}

\begin{proof}[{\bf Proof of Theorem~\ref{p13}.}]
We prove Theorem~\ref{p13} using induction on $n$.
When $n=9$, we first prove the following result:
\[
{\rm ex}(9,\overrightarrow{P_{1,3}})=27\  \text{and \ } {\rm
EX}(9,\overrightarrow{P_{1,3}})=\overrightarrow{\mathcal{S}_9}.
\]

Let $D\in {\rm EX}(9,\overrightarrow{P_{1,3}})$, then $a(D)\geq
a(\overrightarrow{S_9})=27$ since $\overrightarrow{S_9}$ is
$\overrightarrow{P_{1,3}}$-free. By Theorem~\ref{p13small}, we have
${\rm ex}(8,\overrightarrow{P_{1,3}})=22$, which implies that
$d_D(v)\ge 5$ for all $v\in V(D)$. Suppose now that there is $v_0\in
V(D)$ such that $D_1:=D-v_0$ is an extremal digraph in ${\rm
EX}(8,\overrightarrow{P_{1,3}})$, i.e., $D_1$ is the double
orientation of a graph in ${\rm EX}(8,C_4)$. Let $G={\rm Ung}(D)$,
then $G$ has to be $C_4$-free, for otherwise ${\rm Umg}(D)$ contains
a $\varTheta_{1,1,3}$ and $D$ contains a $\overrightarrow{P_{1,3}}$
by Lemma~\ref{LeOrientation}. For convenience, we define ${\rm
Inj}(G)$ as the~\emph{injective graph} of $G$, which is the graph on
$V(G)$ such that $uv\in E({\rm Inj}(G))$ if and only if $N_G(u)\cap
N_G(v)\neq\emptyset$. One can check that every graph in ${\rm
EX}(8,C_4)$ (see Figure~4) satisfies that the independent number of
its injective graph is $2$. Recall that $d_D(v_0)\geq 5$, implying
that $d_G(v_0)\geq 3$ (possibly $v$ is incident with some
$\overrightarrow{C_2}$ in $D$). This implies that $N_G(v_0)$
contains two vertices $v_1,v_2$ such that $v_1,v_2$ have a common
neighbor in ${\rm Ung}(D_1)$. Then $G$ contains a $C_4$, a
contradiction. Thus, we conclude that $a(D-v)<22$ for $v\in V(D)$.
An easy averaging argument shows that there exists $v'\in V(D)$ such
that $d_D(v')\le 6$. If $a(D)=28$, then $a(D-v')\geq 22$, a
contradiction. This implies that $D$ has exactly $27$ arcs and
$d_D(v)=6$ for all $v\in V(D)$.

In the following we will show that $D$ has no
$\overrightarrow{C_2}$. Let $\mathcal{C}$ denote the set of all
$\overrightarrow{C_2}$ in $D$. Suppose first that $|\mathcal{C}|\ge
5$. Let $D_1$ be the subgraph of $D$ by removing arcs of all
$\overrightarrow{C_2}$. Since $a(D)$ is odd, $D_1$ is not empty. Let
$G={\rm Ung}(D)$, $G_1={\rm Ung}(D_1)$, and $G_2$ be the subgraph of
$G$ with $E(G_2)=E(G)\setminus E(G_1)$, then
$e(G_1)=a(D)-2|\mathcal{C}|\le 17$. Note that $G_1$ must be
$K_4$-free. We here show that $G_1$ has a $C_4$-free subgraph $H_1$
with $e(H_1)>e(G_1)/2$. If $G_1$ is disconnected, then by applying
Lemma \ref{LeSubgraph} to each (nontrivial) component of $G_1$, we
can get a required subgraph. If $G_1$ has a vertex $v$ with
$d_{G_1}(v)\le 2$, then by Lemma~\ref{LeSubgraph}, $G_1-v$ has a
subgraph with edge more than $e(G_1-v)/2$, together with an edge
incident to $v$, we get a required subgraphs. Suppose that $G_1$ is
a connected graph with $d_{G_1}(v)\ge 3$ for all $v\in V(G_1)$,
which implies $G_1$ has two cycles of different lengths, one of
which is not a $C_4$. Taking such a cycle, by adding some edges we
can get a $C_4$-free subgraph with $9$ edges. Therefore, $G_1$ has a
$C_4$-free subgraph $H_1$ with $e(H_1)>e(G_1)/2$. Let $H$ be the
subgraph of $G$ with $E(H)=E(H_1)\cup E(G_2)$, then $H$ is a
$C_4$-free graph with $e(H)\geq 14>{\rm ex}(9,C_4)=13$ (see
\cite{ClaphamFlockhartSheehan}), a contradiction. Therefore, we have
$|\mathcal{C}|\le 4$.

Let $U \subset V(D)$ be the set of vertices that not incident to any
$\overrightarrow{C_2}$ in $D$, and $W=V(D)\setminus U$. Since
$|\mathcal{C}|\le 4$, we have $U\neq\emptyset$. Clearly ${\rm
Umg}(D)$ is connected, otherwise $D$ is not an extremal digraph for
$\overrightarrow{P_{1,3}}$. Thus there exists an edge $u_1w_1\in
E({\rm Umg}(D))$ with $u_1\in U$ and $w_1\in W$. Let
$\overrightarrow{C}:=w_1w_2w_1$ be a $\overrightarrow{C_2}$ in $D$
with $w_2\in W$. If $u_1,w_2$ has a common neighbor other than $w_1$
in ${\rm Umg}(D)$, then ${\rm Umg}(D)$ contains a
$\varTheta_{1,1,3}$, a contradiction. Recall that $d_D(v)=6$ for all
$v\in V(D)$. If $u_1w_2\notin E({\rm Umg}(D))$, then $u_1$ has at
least five neighbors other than $w_1$, and $w_2$ has at least two
neighbors other than $w_1$, which implyes that $u_1,w_2$ have a
common neighbor other than $w_1$, a contradiction. Now we conclude
that $u_1w_2\in E({\rm Umg}(D))$. Thus $u_1$ has at least four
neighbors other than $w_1,w_2$. In addition, $w_2$ is contained in
another $\overrightarrow{C_2}$, say $w_2w_3w_2$ with $w_3\in
W\setminus\{w_1\}$. By the similar analysis we have that $u_1w_3\in
E({\rm Umg}(D))$. In other words, $u_1,w_1,w_2,w_3$ are contained in
a $\varTheta_{1,1,3}$, a contradiction. Finally, we conclude that
$D$ is $\overrightarrow{C_2}$-free. Notice that ${\rm Ung}(D)$ is
$K_4$-free with $27$ edges, implying that $D$ is an orientation of
$T_{9,3}$.

Let $\{X,Y,Z\}$ be the balanced partition of ${\rm Umg}(D)$. Fix
arbitrary vertices $x_0\in X,y_0\in Y,z_0\in Z$. By symmetry, we
consider the following two cases: (1) $x_0\rightarrow
y_0,y_0\rightarrow z_0,z_0\rightarrow x_0$; (2) $x_0\rightarrow
z_0,y_0\rightarrow z_0,y_0\rightarrow x_0$. Set
\[\left\{\begin{array}{lll}
    X_1=X\cap N_D^+(z_0), & Y_1=Y\cap N_D^+(x_0), & Z_1=Z\cap N_D^+(y_0),\\
    X_2=X\cap N_D^-(z_0), & Y_2=Y\cap N_D^-(x_0), & Z_2=Z\cap N_D^-(y_0),
\end{array}\right.   \text{ \ for \ } (1)
\]
and
\[\left\{\begin{array}{lll}
    X_1=X\cap N_D^-(z_0), & Y_1=Y\cap N_D^+(x_0), & Z_1=Z\cap N_D^-(y_0),\\
    X_2=X\cap N_D^+(z_0), & Y_2=Y\cap N_D^-(x_0), & Z_2=Z\cap N_D^+(y_0),
\end{array}\right. \text{ \ for \ } (2).
\]
One can check that the partition
$\{X_1,X_2,Y_1,Y_2,Z_1,Z_2\}$ of $V(D)$ satisfies the
Properties~(\ref{P13:a})--(\ref{P13:b}) in the definition of
$\overrightarrow{S_n}$. Therefore, $D\in
\overrightarrow{\mathcal{S}_9}$.

%\end{proof}
%We restate the Theorem for convenience.
%\medskip
%\noindent\textbf{Theorem 1.9.}
%\emph{For all $n\ge 9$,
%$$\dex(n,\overrightarrow{P_{1,3}})=\left\lfloor\frac{n^2}{3}\right\rfloor=\left\{\begin{array}{ll}
%n^2/3,      & \mbox{if } n\equiv 0 \bmod 3;\\
%(n^2-1)/3,  & \mbox{if } n\equiv 1,2 \bmod 3.
%\end{array} \right.$$
%Furthermore, ${\rm{EX}}(n,\overrightarrow{P_{1,3}})=\ria{\mathcal{S}}_{n,3}$.}
%\begin{proof}
Now let $n>9$. For simplicity, set $s(n)=a(\overrightarrow{S_n})$,
and assume that ${\rm ex}(n-1,\overrightarrow{P_{1,3}})=s(n-1)$,
${\rm
EX}(n-1,\overrightarrow{P_{1,3}})=\overrightarrow{\mathcal{S}_{n-1}}$.
We first show that every digraph $D$ on $n$ vertices with
$a(D)=s(n)+1$ contains a $\overrightarrow{P_{1,3}}$. If not, then we must have $d_D(v)\ge \lfloor\frac{2s(n)}{n}\rfloor+1$ for
every $v\in V(D)$: otherwise $D'=D-v$ is a digraph on $n-1$ vertices
with $a(D')\ge s(n-1)+1$, which contains a
$\overrightarrow{P_{1,3}}$. Thus, we have
\[
a(D)=\frac{1}{2}\sum_{v\in V(D)}d_D(v)\ge
\frac{n}{2}\left(\left\lfloor\frac{2s(n)}{n}\right\rfloor+1\right)>
s(n)+1
\]
for $n>9$, which contradicts $a(D)=s(n)+1$. Moreover, since all
digraphs $\overrightarrow{S_n}$ have no $\overrightarrow{P_{1,3}}$,
we get ${\rm ex}(n,\overrightarrow{P_{1,3}})=s(n)$.

Now let $D\in {\rm EX}(n,\overrightarrow{P_{1,3}})$. We will show that
$D\in\overrightarrow{\mathcal{S}_n}$. A simple averaging argument
shows that there exists $v_0\in V(D)$ such that $d_D(v_0)\le
\lfloor\frac{2s(n)}{n}\rfloor$. Set $D':=D-v_0$, then $a(D')\ge
s(n-1)$ (notice that $s(n)-s(n-1)=\lfloor\frac{2s(n)}{n}\rfloor$).
By the induction hypothesis, $D'$ can only be a graph in
$\overrightarrow{\mathcal{S}_{n-1}}$ and $d_D(v_0)=
\lfloor\frac{2s(n)}{n}\rfloor$.
%which implies that $a(D)=s(n)$ and $\lfloor\frac{2s(n)}{n}\rfloor \le d_D(v)\le \lceil\frac{2s(n)}{n}\rceil$ for $v\in V(D)$.

Let $\{X,Y,Z\}$ be the (almost balanced) partition of $V(D')$. If
$v_0$ has neighbors in all the three sets $X,Y,Z$, then ${\rm
Ung}(D)$ contains a $K_4$ and $D$ contains a
$\overrightarrow{P_{1,3}}$, a contradiction. Therefore we assume
without loss of generality that $X\cup\{v_0\}$ is independent. Note
that $|X|\geq \lfloor\frac{n-1}{3}\rfloor$. Let $v_1$ be an
arbitrary vertex in $X$, then
$d_D(v_1)\leq\lceil\frac{2s(n-1)}{n-1}\rceil\leq\lfloor\frac{2s(n)}{n}\rfloor=d_D(v_0)$,
which implies that $D''=D-v_1$ is also a digraph in
$\overrightarrow{\mathcal{S}_{n-1}}$, and thus $D$ is an orientation
of $T_{n,3}$. We claim that one of the following statements holds:
(1) for every $u\in Y\cup Z$, $u\rightarrow v_1$ if and only if
$u\rightarrow v_0$; (2) for every $u\in Y\cup Z$, $u\rightarrow v_1$
if and only if $v_0\rightarrow u$. Suppose that neither (1) nor (2)
holds. Without loss of generality there are two vertices $u_0,u_1$,
such that $u_0\rightarrow v_1$, $u_0\rightarrow v_0$,
$u_1\rightarrow v_1$ and $v_1\rightarrow u_2$, then
$D[\{v_0,v_1,u_0,u_1\}]$ contains a $\overrightarrow{P_{1,3}}$, a
contradiction. Therefore, $D\in\overrightarrow{\mathcal{S}_n}$.
\end{proof}

Next we will prove Theorem~\ref{anticycle} using the transformation between digraph and balanced bipartite graph.
%, and Hall's theorem.
Given a labeled balanced bipartite graph $B=(X,Y;E)$ with $X=\{x_i:
i\in[1,n]\}, Y=\{y_i: i\in[1,n]\}$, let ${\rm Drp}(B)$ denote the
digraph with vertex set $V=\{v_i: i\in[1,n]\}$ and arc set
$A=\{v_iv_j: x_iy_j\in E\}$. We call ${\rm Drp}(B)$ the
\textit{directed representation} of $B$. We notice that we also
defined the bipartite representation ${\rm Brp}(D)$ of a digraph $D$
(see the definition before Theorem~\ref{anticycle}). Clearly,
$e({\rm Brp}(D))=a(D)$ for all digraph $D$; and if a labeled
bipartite graph $B$ satisfies $x_iy_i\notin E(B)$ for all
$i\in[1,n]$, then $a({\rm Drp}(B))=e(B)$.
%(see Figure 7)
%\begin{center}
%\begin{picture}(130,90)\label{bipartite-rep.}
%\thicklines \put(0,10){ \put(0,10){\circle*{3}}
%\put(0,40){\circle*{3}} \put(40,40){\circle*{3}}
%\put(40,10){\circle*{3}} \put(20,60){\circle*{3}}
%
%\put(17.5,62.5){$v_1$} \put(-7.5,45){$v_2$} \put(37.5,45){$v_3$}
%\put(-12.5,10){$v_4$} \put(42.5,10){$v_5$}
%
%\put(0,10){\line(1,0){40}} \put(0,10){\line(0,1){30}}
%\put(40,40){\line(-1,0){40}} \put(40,40){\line(0,-1){30}}
%\put(0,40){\line(1,1){20}} \put(40,40){\line(-1,1){20}}
%
%\put(12,52){\vector(-1,-1){5}} \put(32,48){\vector(-1,1){5}}
%\put(0,25){\vector(0,-1){5}} \put(20,10){\vector(-1,0){5}}
%\put(40,25){\vector(0,1){5}} \put(20,40){\vector(1,0){5}}
%\put(13,-7){$D$}}
%
%\put(90,10){ \put(0,10){\circle*{3}} \put(0,25){\circle*{3}}
%\put(0,40){\circle*{3}} \put(0,55){\circle*{3}}
%\put(0,70){\circle*{3}}
%
%\put(30,10){\circle*{3}} \put(30,25){\circle*{3}} \put(30,40){\circle*{3}}
%\put(30,55){\circle*{3}} \put(30,70){\circle*{3}}
%
%\put(0,70){\line(2,-1){30}} \put(0,55){\line(2,-1){30}} \put(0,40){\line(1,1){30}}
%\put(0,55){\line(1,-1){30}} \put(0,10){\line(1,1){30}}  \put(0,10){\line(2,1){30}}
%
%\put(-12.5,10){$x_5$} \put(-12.5,24){$x_4$} \put(-12.5,38){$x_3$}
%\put(-12.5,52){$x_2$} \put(-12.5,66){$x_1$}
%
%\put(32.5,10){$y_5$} \put(32.5,24){$y_4$} \put(32.5,38){$y_3$}
%\put(32.5,52){$y_2$} \put(32.5,66){$y_1$}
%
%\put(0,-7){${\rm Brp}(D)$} }
%\end{picture}
%
%\small Figure~7. A digraph and its bipartite representation.
%\end{center}

For $k\in \mathbb{N}^*$, an orientation of $C_{2k}$ is called
\textit{anti-directed} if the orientation of each arc is opposite to
that of its predecessor, denoted by $\overrightarrow{C_{2k}^a}$.
Note that $\overrightarrow{K_{2,2}}=\overrightarrow{C_4^a}$. It is
not hard to see that every $\overrightarrow{C_{2k}^a}$ in a digraph
$D$ corresponds to a $C_{2k}$ in its bipartite representation ${\rm
Brp}(D)$.

Recall that ${\rm ex}(n,n; C_{2k})$ denote the maximum number of
edges in a $C_{2k}$-free bipartite graph with parts of equal size
$n$. By the construction, it is easy to obtain the following
observation.

\begin{observation}\label{anti-lemma}
For $n,k\in \mathbb{N}^*$ with $k\ge 2$,
\[
{\rm ex}(n,n;C_{2k})-n\le {\rm ex}(n,\overrightarrow{C_{2k}^a})\le
{\rm ex}(n,n;C_{2k}).
\]
\end{observation}

%\begin{proof}
%Suppose that $D$ is a $\overrightarrow{C_{2k}^a}$-free digraph on $n$ vertices with $\ria{ex}(n,\overrightarrow{C_{2k}^a})$ arcs. By digraph $D$, we construct its bipartite representation $BG(D)$. Since $D$ is $\overrightarrow{C_{2k}^a}$-free, $BG(D)$ is $C_{2k}$-free. Hence, $a(D)=e(BG(D))\le ex(n,n;C_{2k})$. Next, let $B$ is a $C_{2k}$-free balanced bipartite graph on $2n$ vertices with $\ria{ex}(n,n;C_{2k})$ edges. By digraph $D$, we construct its directed representation $D(B)$. Since $B$ is $C_{2k}$-free, $D(B)$ is $\overrightarrow{C_{2k}^a}$-free digraph on $n$ vertices. In particular, $a(D(B)\ge e(B)-n$.
%\end{proof}
A \textit{perfect matching} in a graph is a collection of edges such
that every vertex is contained in exactly one of them. In the proof
of the Theorem~\ref{anticycle}, we will use a famous Hall's theorem.

\begin{theorem}[Hall~\cite{MR3822066}]\label{Hall}
Let $B$ be a bipartite graph with parts $X$ and $Y$ of equal size.
If for every $U\subseteq X$, $|N_B(U)|\ge |U|$, then $B$ contains a
perfect matching.
\end{theorem}

\begin{proof}[{\bf Proof of Theorem \ref{anticycle}.}]
For every digraph $D$ and the bipartite representation ${\rm
Brp}(D)$, $D$ is $\overrightarrow{K_{2,2}}$-free if and only if
${\rm Brp}(D)$ is $C_4$-free. This implies that ${\rm
ex}(n,\overrightarrow{K_{2,2}})\leq{\rm ex}(n,n;C_4)$.

Now let $B=(X,Y;E)$ be an extremal bipartite graph in ${\rm
EX}(n,n;C_4)$ with $n\geq 3$. We claim that there is a label on
$V(B)$, say $X=\{x_i: i\in[n]\}$, $Y=\{y_i: i\in[n]\}$, such that
$x_iy_i\notin E$. Let $\widehat{B}$, the \emph{quasi-complement} of
$B$, be the bipartite graph on $V(B)$ such that $xy\in
E(\widehat{B})$ if and only if $xy\notin E(B)$ for any $x\in X,y\in
Y$. We show that $\widehat{B}$ has a perfect matching.

Suppose that $\widehat{B}$ does not have perfect matching. By
Theorem \ref{Hall}, there is a subset $U\subseteq X$ such that
$|N_{\widehat{B}}(U)|<|U|$. Since $|U|\le n$, we have
$|N_{\widehat{B}}(U)|\le n-1$. If $|N_{\widehat{B}}(U)|=n-1$, then
$U=X$, which implies that there is one vertex $y'\in Y$ such that
$N_{\widehat{B}}(y')=\emptyset$. Thus $N_{B}(y')=X$. Since $B$ is
$C_4$-free, we have $|N_{B}(y)|\le 1$ for every $y\in Y\setminus
\{y'\}$, implying that $e(B)\leq 2n-1$. Note that $C_{2n}$ is a
$C_4$-free bipartite graph for $n\ge 3$, implying $e(B)<{\rm
ex}(n,n;C_4)$, a contradiction. If $|U|=1$, then
$N_{\widehat{B}}(U)=\emptyset$, and we can analysis as above. Now
assume that $|U|\geq 2$ and $|N_{\widehat{B}}(U)|<n-1$, then there
are at least two vertices $y_1,y_2\in Y\setminus
N_{\widehat{B}}(U)$. This implies that $U\subseteq N_B(y_1)\cap
N_B(y_2)$. Since $|U|\ge 2$, $B$ contains a $C_4$, a contradiction.
Thus we conclude that $\widehat{B}$ has a perfect matching. Label
the vertices of $B$ such that $\{x_iy_i: i\in[1,n]\}$ is a perfect
matching of $\widehat{B}$. The directed representation ${\rm
Drp}(B)$ is a $\overrightarrow{K_{2,2}}$-free digraph with $a({\rm
Drp}(B))={\rm ex}(n,n;C_4)$. Thus we have ${\rm
ex}(n,\overrightarrow{K_{2,2}})\geq{\rm ex}(n,n;C_4)$.

Conversely, for every $\overrightarrow{K_{2,2}}$-free extremal
digraph $D$, we can get a $C_4$-free graph ${\rm Brp}(D)$ such that
$e({\rm Brp}(D))=a(D)={\rm ex}(n,n;C_4)$. Therefore, we have
${\rm{EX}}(n,\overrightarrow{K_{2,2}})=\{D: {\rm Brp}(D)\in
{\rm{EX}}(n,n; C_4)\}$ for $n\ge 3$.
\end{proof}

To finish this section, we shall prove Lemma \ref{LeSubgraph}. We
make use of the following results.

\begin{theorem}[Fan \cite{Fan}]\label{Fan}
  Let $G$ be a $2$-connected graph of order $n$. If for every two vertices
  $u,v$ with distance $2$, $\max\{d(u),d(v)\}\geq n/2$, then $G$ contains a Hamilton cycle.
\end{theorem}

\begin{theorem}[Gy\H{o}ri, Keszegh \cite{GyoriKeszegh}]\label{GyoriKeszegh}
  Every $K_4$-free graph with $n$ vertices and ${\rm ex}(n,K_3)+t$ edges
  contains at least $t$ edge-disjoint triangles.
\end{theorem}

A graph $G$ is \emph{perfect} if for every induced subgraph $H$ of
$G$, $\chi(H)=\omega(H)$, where $\chi(H)$ is the chromatic number of
$H$ and $\omega(H)$ is the clique number of $H$.

\begin{theorem}[Chudnovsky et al. \cite{Chudnovsky}]\label{Perfect}
A graph is perfect if and only if it contains no odd cycle of length
at least five, or its complement, as an induced subgraph.
\end{theorem}

Now we are ready to prove Lemma~\ref{LeSubgraph}. A \emph{chord} of
a cycle is an edge connecting two vertices of the cycle but not an
edge of the cycle. Given a cycle $C$ of a graph $G$, for $i\in
\mathbb{N}^*$, an \textit{$i$-chord} of $C$ is a chord $uv$ such
that $v$ and $u$ has distance $i$ in $C$; an \textit{$i$-bridge} of
$C$ is a path $P:=uwv$ with $u,v\in V(C)$, $w\notin V(C)$ and $v, u$
has distance $i$ in $C$.

\begin{proof}[{\bf Proof of Lemma \ref{LeSubgraph}.}]
For every graph $G$ on $n\le 3$ vertices, $G$ clearly satisfies the
result, since $G$ is both $K_4$-free and $C_4$-free. Suppose now
there are some $K_4$-free graphs on $4\le n\le 8$ vertices that do
not satisfy the result, and choose a counterexample graph $G$ with
minimal order. First we claim $G$ is $2$-connected and $d_G(v)\ge 3$
for all $v\in V(G)$. Indeed, if $G$ is not 2-connected, then each
(nontrivial) block $G_i$ of $G$ contains a $C_4$-free subgraph $H_i$
with $e(H_i)>e(G_i)/2$. Thus the union of all such $H_i$ form a
required subgraph of $G$. Second, if there is $v\in V(G)$ with
$d_G(v)=2$, then $G_1:=G-v$ has a $C_4$-free subgraph $H_1$ with
$e(H_1)>e(G_1)/2=e(G)/2-1$. Thus the graph $H$ consisting of $H_1$
and one edge incident with $v$ is a required subgraph of $G$.

Next, by considering a longest path and the neighborhood of its
end-vertices, one can see that $G$ has two cycles of different
lengths, one of which is not a $C_4$. Taking such a cycle by adding
edges we can get a $C_4$-free subgraph $H$ of $G$ with $e(H)=n$,
which implies that $e(G)\geq 2n$. Recall that $G$ is $K_4$-free, by
Theorem \ref{K_k-free}, we infer that the counterexample graph $G$
can only be in the following cases: $n=6$ and $e(G)=12$; $n=7$ and
$e(G)\in[14,16]$; or $n=8$ and $e(G)\in[16,21]$.

If $n=6$ and $e(G)=12$, then $G$ can only be the Tur\'an graph
$T_{6,3}$. Clearly $G$ has a $C_4$-free subgraph with $7$ edges (see
Figure~7). Similarly, if $n=7$ and $e(G)=16$, then $G$ is also the
Tur\'an graph $T_{7,3}$. One can check that $G$ has a $C_4$-free
subgraph with $9$ edges (see Figure~7). For $n=7$ and $e(G)=14$ or
$15$, by Theorem \ref{GyoriKeszegh} $G$ has two edge-disjoint
triangles. Taking such two triangles by adding edges we can get a
$C_4$-free subgraph with 8 edges. So in the following we consider
the case $n=8$.

\begin{center}
\begin{picture}(380,120)
\put(70,60){\put(0,-50){\circle*{4}} \put(45,-25){\circle*{4}}
\put(45,25){\circle*{4}} \put(0,50){\circle*{4}}
\put(-45,25){\circle*{4}} \put(-45,-25){\circle*{4}}
\put(0,-50){\line(9,5){45}} \put(45,-25){\line(0,1){50}}
\put(45,25){\line(-9,5){45}} \put(0,50){\line(-9,-5){45}}
\put(-45,25){\line(0,-1){50}} \put(-45,-25){\line(9,-5){45}}
\put(0,-50){\line(3,5){45}} \put(45,-25){\line(-3,5){45}}
\put(45,25){\line(-1,0){90}} \put(0,50){\line(-3,-5){45}}
\put(-45,25){\line(3,-5){45}} \put(-45,-25){\line(1,0){90}}
\textcolor{blue}{\thicklines \put(0,-50){\line(9,5){45}}
\put(45,25){\line(-9,5){45}} \put(0,50){\line(-9,-5){45}}
\put(-45,25){\line(0,-1){50}} \put(-45,-25){\line(9,-5){45}}
\put(45,25){\line(-1,0){90}} \put(-45,-25){\line(1,0){90}} } }

\put(190,60){\put(0,0){\circle*{4}} \put(0,-50){\circle*{4}}
\put(45,-25){\circle*{4}} \put(45,25){\circle*{4}}
\put(0,50){\circle*{4}} \put(-45,25){\circle*{4}}
\put(-45,-25){\circle*{4}} \put(0,-50){\line(9,5){45}}
\put(45,-25){\line(0,1){50}} \put(45,25){\line(-9,5){45}}
\put(0,50){\line(-9,-5){45}} \put(-45,25){\line(0,-1){50}}
\put(-45,-25){\line(9,-5){45}} \put(0,-50){\line(3,5){45}}
\put(45,-25){\line(-3,5){45}} \put(45,25){\line(-1,0){90}}
\put(0,50){\line(-3,-5){45}} \put(-45,25){\line(3,-5){45}}
\put(-45,-25){\line(1,0){90}} \put(0,0){\line(9,5){45}}
\put(0,0){\line(-9,5){45}} \put(0,0){\line(-9,-5){45}}
\put(0,0){\line(9,-5){45}} \textcolor{blue}{\thicklines
\put(0,-50){\line(9,5){45}} \put(45,25){\line(-9,5){45}}
\put(0,50){\line(-9,-5){45}} \put(-45,25){\line(0,-1){50}}
\put(-45,-25){\line(9,-5){45}} \put(45,25){\line(-1,0){90}}
\put(-45,-25){\line(1,0){90}} \put(0,0){\line(-9,5){45}}
\put(0,0){\line(-9,-5){45}} } }

\put(310,60){\put(0,-50){\circle*{4}} \put(0,50){\circle*{4}}
\put(50,0){\circle*{4}} \put(-50,0){\circle*{4}}
\put(-35,-35){\circle*{4}} \put(-35,35){\circle*{4}}
\put(35,-35){\circle*{4}} \put(35,35){\circle*{4}}
\put(-50,0){\line(1,0){100}} \put(50,0){\line(-3,7){15}}
\put(35,35){\line(-7,3){35}} \put(0,50){\line(-7,-3){35}}
\put(-35,35){\line(-3,-7){15}} \put(-35,-35){\line(-3,7){15}}
\put(-35,-35){\line(0,1){70}} \put(-35,-35){\line(7,17){35}}
\put(-35,-35){\line(1,1){70}} \put(-35,-35){\line(17,7){85}}
\put(0,-50){\line(-1,1){50}} \put(0,-50){\line(-7,17){35}}
\put(0,-50){\line(0,1){100}} \put(0,-50){\line(7,17){35}}
\put(0,-50){\line(1,1){50}} \put(35,-35){\line(-17,7){85}}
\put(35,-35){\line(-1,1){70}} \put(35,-35){\line(-7,17){35}}
\put(35,-35){\line(0,1){70}} \put(35,-35){\line(3,7){15}}
\textcolor{blue}{\thicklines \put(-50,0){\line(1,0){100}}
\put(50,0){\line(-3,7){15}} \put(35,35){\line(-7,3){35}}
\put(0,50){\line(-7,-3){35}} \put(-35,35){\line(-3,-7){15}}
\put(-35,-35){\line(-3,7){15}} \put(-35,-35){\line(0,1){70}}
\put(0,-50){\line(-7,17){35}} \put(0,-50){\line(0,1){100}}
\put(35,-35){\line(-7,17){35}} \put(35,-35){\line(0,1){70}}} }
\end{picture}

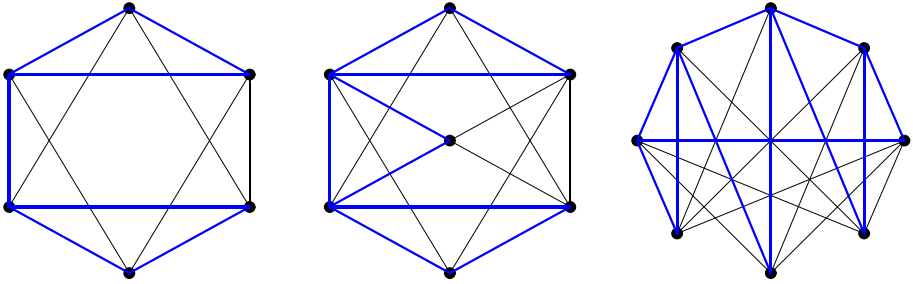
\small Figure~7. Graphs $T_{6,3}$, $T_{7,3}$,
$C_5\vee\overline{K_3}$, and their desired $C_4$-free subgraphs.
\end{center}

\underline{Case A. $e(G)=20$ or 21}. In this case we need find a
$C_4$-free subgraphs with 11 edges. We only need to deal with the
case $e(G)=20$. We claim that $G$ can not contain induced $C_5$,
$C_7$ and $\overline{C_7}$. If $G$ has an induced $C_7$, then
$e(G)\leq 7+7<20$, a contradiction. If $G$ has an induced
$\overline{C_7}$, then $d_G(v)\le 4$ for the vertex $v$ outside the
$\overline{C_7}$; for otherwise $G$ contains a $K_4$. It follows
that $e(G)\leq 14+4<20$, a contradiction. Furthermore, suppose that
$G$ has an induced $C_5$, denoted by $C$. Set $V':=V(G)\setminus
V(C)$. If $e(G[V'])=3$, then $|N_G(v)\cap V'|\le 2$ for all $v\in
V(C)$, for otherwise $G$ contains a $K_4$. It follows that $e(G)\leq
5+10+3<20$, a contradiction. Suppose now $e(G[V'])=1$ or 2. Let
$v_1v_2\in E(G)$ with $v_1,v_2\in V'$. Then $|N_G(v_1)\cap
V(C)|+|N_G(v_2)\cap V(C)|\leq 7$; for otherwise $G$ has a $K_4$.
This implies that $e(G)\leq 5+2+7+5<20$, a contradiction. Now we
suppose $e(G[V'])=0$, then $e(G)\leq 5+15=20$. This implies that $G$
is the join of $C_5$ and $\overline{K_3}$. One can check that $G$
has a $C_4$-free subgraph with $11$ edges (see Figure~7). Now we
conclude that $G$ has no induced copies of $C_5,C_7$ or
$\overline{C_7}$.
%If there is a vertex outside $C$ that has at most two edges to $C$, then
%$e(G)=e(F)+e(F,G-F)+e(G-F)\leq 5+12+3=20$. It follows that $G-C$ is
%a triangle and there is a vertex neighboring all the three vertices
%of the triangle. Thus $G$ contains a $K_4$, a contradiction. If
%there are two vertices outside $C$ both have three edges to $C$,
%then $e(G)\leq 5+11+3<20$, a contradiction. So we conclude that one
%vertex in $G-F$ has at least three neighbors in $C$, and the others
%have at least four. Thus $G$ has subgraph $H$ consisting of the
%five-cycle $C$ and three triangles each has an edge common to $C$,
%which is $C_4$-free and has 11 edges.

By Theorem~\ref{Perfect}, we obtain that $G$ is a perfect graph.
Recall that $G$ is also $K_4$-free. We have $\chi(G)\leq 3$, i.e.,
$G$ is a subgraph of a complete $3$-partite graph $K_{r,s,t}$ with
$r+s+t=8$. Note that $e(G)=20\leq e(K_{r,s,t})=rs+rt+st$. One can
compute that $K_{r,s,t}$ can only be $K_{2,3,3}$ or $K_{2,2,4}$.
Thus $G$ is either $K_{2,2,4}$ or one graph obtained from
$K_{2,3,3}$ by removing an edge (there are two non-isomorphic such
graphs). One can check that $G$ has a $C_4$-free subgraph with $11$
edges for each case (see Figure~8).

\begin{center}
\begin{picture}(380,100)\label{P13}

\put(20,10){\put(0,20){\circle*{4}} \put(20,0){\circle*{4}}
\put(80,0){\circle*{4}} \put(100,20){\circle*{4}}
\put(20,80){\circle*{4}} \put(40,80){\circle*{4}}
\put(60,80){\circle*{4}} \put(80,80){\circle*{4}}
\put(0,20){\line(1,0){100}} \put(0,20){\line(4,-1){80}}
\put(20,0){\line(1,0){60}} \put(20,0){\line(4,1){80}}
\put(0,20){\line(1,3){20}} \put(0,20){\line(2,3){40}}
\put(0,20){\line(1,1){60}} \put(0,20){\line(4,3){80}}
\put(20,0){\line(0,1){80}} \put(20,0){\line(1,4){20}}
\put(20,0){\line(1,2){40}} \put(20,0){\line(3,4){60}}
\put(100,20){\line(-1,3){20}} \put(100,20){\line(-2,3){40}}
\put(100,20){\line(-1,1){60}} \put(100,20){\line(-4,3){80}}
\put(80,0){\line(0,1){80}} \put(80,0){\line(-1,4){20}}
\put(80,0){\line(-1,2){40}} \put(80,0){\line(-3,4){60}}
\textcolor{blue}{\thicklines \put(0,20){\line(4,-1){80}}
\put(20,0){\line(1,0){60}} \put(20,0){\line(4,1){80}}
\put(0,20){\line(1,3){20}} \put(0,20){\line(2,3){40}}
\put(20,0){\line(1,2){40}} \put(20,0){\line(3,4){60}}
\put(100,20){\line(-1,3){20}} \put(100,20){\line(-4,3){80}}
\put(80,0){\line(-1,4){20}} \put(80,0){\line(-1,2){40}} }}

\put(140,10){\put(0,40){\circle*{4}} \put(10,20){\circle*{4}}
\put(20,0){\circle*{4}} \put(80,0){\circle*{4}}
\put(90,20){\circle*{4}} \put(100,40){\circle*{4}}
\put(35,80){\circle*{4}} \put(65,80){\circle*{4}}
\put(0,40){\line(1,0){100}} \put(0,40){\line(9,-2){90}}
\put(0,40){\line(2,-1){80}} \put(10,20){\line(9,2){90}}
\put(10,20){\line(1,0){80}} \put(10,20){\line(7,-2){70}}
\put(20,0){\line(2,1){80}} \put(20,0){\line(7,2){70}}
\put(35,80){\line(-7,-8){35}} \put(35,80){\line(-5,-12){25}}
\put(35,80){\line(-3,-16){15}} \put(35,80){\line(9,-16){45}}
\put(35,80){\line(11,-12){55}} \put(35,80){\line(13,-8){65}}
\put(65,80){\line(-13,-8){65}} \put(65,80){\line(-11,-12){55}}
\put(65,80){\line(-9,-16){45}} \put(65,80){\line(3,-16){15}}
\put(65,80){\line(5,-12){25}} \put(65,80){\line(7,-8){35}}
\textcolor{blue}{\thicklines \put(0,40){\line(2,-1){80}}
\put(10,20){\line(7,-2){70}} \put(20,0){\line(2,1){80}}
\put(20,0){\line(7,2){70}} \put(35,80){\line(-5,-12){25}}
\put(35,80){\line(-3,-16){15}} \put(35,80){\line(9,-16){45}}
\put(35,80){\line(11,-12){55}} \put(65,80){\line(-13,-8){65}}
\put(65,80){\line(-9,-16){45}} \put(65,80){\line(7,-8){35}} }
\put(10,0){$u$} \put(85,0){$v$} }

\put(260,10){\put(0,40){\circle*{4}} \put(10,20){\circle*{4}}
\put(20,0){\circle*{4}} \put(80,0){\circle*{4}}
\put(90,20){\circle*{4}} \put(100,40){\circle*{4}}
\put(35,80){\circle*{4}} \put(65,80){\circle*{4}}
\put(0,40){\line(1,0){100}} \put(0,40){\line(9,-2){90}}
\put(0,40){\line(2,-1){80}} \put(10,20){\line(9,2){90}}
\put(10,20){\line(1,0){80}} \put(10,20){\line(7,-2){70}}
\put(20,0){\line(2,1){80}} \put(20,0){\line(7,2){70}}
\put(20,0){\line(1,0){60}} \put(35,80){\line(-5,-12){25}}
\put(35,80){\line(-3,-16){15}} \put(35,80){\line(9,-16){45}}
\put(35,80){\line(11,-12){55}} \put(35,80){\line(13,-8){65}}
\put(65,80){\line(-13,-8){65}} \put(65,80){\line(-11,-12){55}}
\put(65,80){\line(-9,-16){45}} \put(65,80){\line(3,-16){15}}
\put(65,80){\line(5,-12){25}} \put(65,80){\line(7,-8){35}}
\textcolor{blue}{\thicklines \put(0,40){\line(2,-1){80}}
\put(10,20){\line(7,-2){70}} \put(20,0){\line(2,1){80}}
\put(20,0){\line(7,2){70}} \put(35,80){\line(-5,-12){25}}
\put(35,80){\line(-3,-16){15}} \put(35,80){\line(9,-16){45}}
\put(35,80){\line(11,-12){55}} \put(65,80){\line(-13,-8){65}}
\put(65,80){\line(-9,-16){45}} \put(65,80){\line(7,-8){35}} }
\put(25,77.5){$u$} \put(-2.5,45){$v$}}
\end{picture}

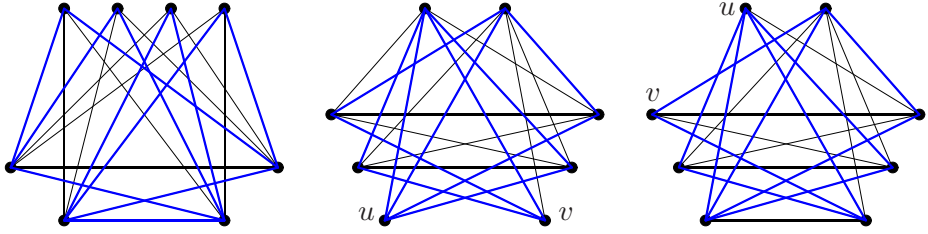
\small Figure 8. Graph $K_{2,2,4}$ and the graphs obtained from
$K_{2,3,3}$ by removing one edge $uv$.
\end{center}

\underline{Case B. $e(G)=18$ or $19$}. In this case we need find a
$C_4$-free subgraphs with $10$ edges. We only deal with the case
$e(G)=18$. We first show that {$G$ contains a Hamilton cycle.}
Recall that $d_G(v)\ge 3$ for all $v\in G$. If there are two
non-adjacent vertices $v_1,v_2\in V(G)$ such that
$d_G(v_1)=d_G(v_2)=3$, then $F:=G[V(G)\setminus \{v_1,v_2\}]$ has
$6$ vertex and $12$ edges, which can only be a $T_{6,3}$. It follows
that $G[N_G(v_i)]$ has at least two edges for $i\in [1,2]$. Let
$e_1\in E(G[N_G(v_1)])$ and $e_2\in E(G[N_G(v_2)])$ such that
$e_1\neq e_2$. {It is easy to check that $T_{6,3}$ has a Hamilton
cycle passing through any two edges (see Figure 9).} Let $C'$ be a
Hamilton cycle of $F$ passing through $e_1,e_2$. By replacing
$e_1,e_2$ with two 1-bridges of $C'$ passing through $v_1,v_2$,
respectively, we can get a Hamilton cycle of $G$. If for each two
non-adjacent vertices $u,v\in V(G)$, $\max\{d(u),d(v)\}\geq 4$, then
$G$ also contains a Hamilton cycle by Theorem~\ref{Fan}.

\begin{center}
\begin{picture}(420,80)

\put(20,10){\put(0,20){\circle*{4}} \put(20,0){\circle*{4}}
\put(50,0){\circle*{4}} \put(70,20){\circle*{4}}
\put(20,60){\circle*{4}} \put(50,60){\circle*{4}}
\put(0,20){\line(1,0){70}} \put(0,20){\line(5,-2){50}}
\put(20,0){\line(5,2){50}} \put(20,0){\line(1,0){30}}
\put(20,60){\line(-1,-2){20}} \put(20,60){\line(0,-1){60}}
\put(20,60){\line(1,-2){30}} \put(20,60){\line(5,-4){50}}
\put(50,60){\line(-5,-4){50}} \put(50,60){\line(-1,-2){30}}
\put(50,60){\line(0,-1){60}} \put(50,60){\line(1,-2){20}}
\textcolor{blue}{\thicklines \put(20,60){\line(0,-1){60}}
\put(20,60){\line(5,-4){50}} \put(50,60){\line(-5,-4){50}}
\put(50,60){\line(0,-1){60}} \put(0,20){\line(1,0){70}}
\put(20,0){\line(1,0){30}} } \put(7,0){$u_1$} \put(52,0){$v_1$}
\put(-10,25){$u_2$} \put(70,25){$v_2$} }

\put(120,10){\put(0,20){\circle*{4}} \put(20,0){\circle*{4}}
\put(50,0){\circle*{4}} \put(70,20){\circle*{4}}
\put(20,60){\circle*{4}} \put(50,60){\circle*{4}}
\put(0,20){\line(1,0){70}} \put(0,20){\line(5,-2){50}}
\put(20,0){\line(5,2){50}} \put(20,0){\line(1,0){30}}
\put(20,60){\line(-1,-2){20}} \put(20,60){\line(0,-1){60}}
\put(20,60){\line(1,-2){30}} \put(20,60){\line(5,-4){50}}
\put(50,60){\line(-5,-4){50}} \put(50,60){\line(-1,-2){30}}
\put(50,60){\line(0,-1){60}} \put(50,60){\line(1,-2){20}}
\textcolor{blue}{\thicklines \put(20,60){\line(-1,-2){20}}
\put(20,60){\line(0,-1){60}} \put(50,60){\line(0,-1){60}}
\put(50,60){\line(1,-2){20}} \put(0,20){\line(1,0){70}}
\put(20,0){\line(1,0){30}} } \put(-20,55){$u_1=u_2$}
\put(7,0){$v_1$} \put(-10,25){$v_2$} }

\put(220,10){\put(0,20){\circle*{4}} \put(20,0){\circle*{4}}
\put(50,0){\circle*{4}} \put(70,20){\circle*{4}}
\put(20,60){\circle*{4}} \put(50,60){\circle*{4}}
\put(0,20){\line(1,0){70}} \put(0,20){\line(5,-2){50}}
\put(20,0){\line(5,2){50}} \put(20,0){\line(1,0){30}}
\put(20,60){\line(-1,-2){20}} \put(20,60){\line(0,-1){60}}
\put(20,60){\line(1,-2){30}} \put(20,60){\line(5,-4){50}}
\put(50,60){\line(-5,-4){50}} \put(50,60){\line(-1,-2){30}}
\put(50,60){\line(0,-1){60}} \put(50,60){\line(1,-2){20}}
\textcolor{blue}{\thicklines \put(20,60){\line(0,-1){60}}
\put(20,60){\line(5,-4){50}} \put(50,60){\line(-5,-4){50}}
\put(50,60){\line(0,-1){60}} \put(0,20){\line(1,0){70}}
\put(20,0){\line(1,0){30}} } \put(-20,55){$u_1=u_2$}
\put(7,0){$v_1$} \put(70,25){$v_2$} }

\put(320,10){\put(0,20){\circle*{4}} \put(20,0){\circle*{4}}
\put(50,0){\circle*{4}} \put(70,20){\circle*{4}}
\put(20,60){\circle*{4}} \put(50,60){\circle*{4}}
\put(0,20){\line(1,0){70}} \put(0,20){\line(5,-2){50}}
\put(20,0){\line(5,2){50}} \put(20,0){\line(1,0){30}}
\put(20,60){\line(-1,-2){20}} \put(20,60){\line(0,-1){60}}
\put(20,60){\line(1,-2){30}} \put(20,60){\line(5,-4){50}}
\put(50,60){\line(-5,-4){50}} \put(50,60){\line(-1,-2){30}}
\put(50,60){\line(0,-1){60}} \put(50,60){\line(1,-2){20}}
\textcolor{blue}{\thicklines \put(20,60){\line(-1,-2){20}}
\put(20,60){\line(0,-1){60}} \put(50,60){\line(0,-1){60}}
\put(50,60){\line(1,-2){20}} \put(0,20){\line(1,0){70}}
\put(20,0){\line(1,0){30}} } \put(5,55){$u_1$} \put(55,55){$u_2$}
\put(-10,25){$v_1$} \put(70,25){$v_2$} }
\end{picture}

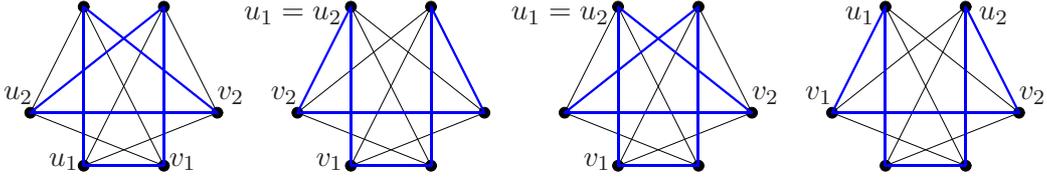
\small Figure 9. Hamilton cycles of $K_{2,2,2}$ passing through two
edges $u_1v_1,u_2v_2$.
\end{center}

Let now $C$ be a Hamilton cycle of $G$. Note that $C$ has at most
eight $2$-chords, eight $3$-chords and four $4$-chords. If $C$ has
three $2$-chords, then two of them together with $C$ form a
$C_4$-free subgraph of $G$ with $10$ edges. If $C$ has three
$4$-chords, then two of them together $C$ form a $C_4$-free subgraph
of $G$ with $10$ edges. Therefore, $G$ has at most two $2$-chords
and at most two $4$-chords. In addition, since $e(G)=18$, $C$ has at
least six $3$-chords. Thus $C$ together with six $3$-chords form a
subgraph of $K_{4,4}$ with $14$ edges. Note there are two
non-isomorphic subgraph of $K_{4,4}$ with $14$ edges, each of which
has a $\varTheta_{3,3,3}$-subgraph (see Figure~10). Let $F$ be a
$\varTheta_{3,3,3}$-subgraph of $G$ and $e=v_1v_2$ be an arbitrary
edge in $E(G)\setminus E(F)$. If $v_1,v_2$ have distance in $F$ not
equal to $3$, then the subgraph of $G$ obtained from $F$ by adding
the edge $e$ is a $C_4$-free graph with $10$ edges. It follows that
every edge in $E(G)\setminus E(F)$ has its two end-vertices with
distance $3$ in $F$. Note that there are exactly $7$ pairs of
vertices with distance $3$ in $F$. Thus we have $e(G)\leq 9+7<18$, a
contradiction.

\begin{center}
\begin{picture}(200,70)
\put(20,10){\put(0,5){\circle*{4}} \put(0,45){\circle*{4}}
\put(20,5){\circle*{4}} \put(20,45){\circle*{4}}
\put(40,5){\circle*{4}} \put(40,45){\circle*{4}}
\put(60,5){\circle*{4}} \put(60,45){\circle*{4}}
\put(0,5){\line(1,1){40}} \put(0,5){\line(3,2){60}}
\put(20,5){\line(-1,2){20}} \put(20,5){\line(1,2){20}}
\put(20,5){\line(1,1){40}} \put(40,5){\line(-1,1){40}}
\put(40,5){\line(-1,2){20}} \put(40,5){\line(0,1){40}}
\put(40,5){\line(1,2){20}} \put(60,5){\line(-3,2){60}}
\put(60,5){\line(-1,1){40}} \put(60,5){\line(-1,2){20}}
\put(60,5){\line(0,1){40}} \textcolor{blue}{\thicklines
\put(0,5){\line(1,2){20}} \put(0,5){\line(3,2){60}}
\put(20,5){\line(-1,2){20}} \put(20,5){\line(1,1){40}}
\put(40,5){\line(0,1){40}} \put(40,5){\line(1,2){20}}
\put(60,5){\line(-3,2){60}} \put(60,5){\line(-1,1){40}}
\put(60,5){\line(-1,2){20}} } \put(-2.5,-2.5){$u_1$}
\put(-2.5,48.5){$v_1$} \put(17.5,-2.5){$u_2$} \put(17.5,48.5){$v_2$}
}

\put(120,10){\put(0,5){\circle*{4}} \put(0,45){\circle*{4}}
\put(20,5){\circle*{4}} \put(20,45){\circle*{4}}
\put(40,5){\circle*{4}} \put(40,45){\circle*{4}}
\put(60,5){\circle*{4}} \put(60,45){\circle*{4}}
\put(0,5){\line(1,1){40}} \put(0,5){\line(3,2){60}}
\put(20,5){\line(-1,2){20}} \put(20,5){\line(0,1){40}}
\put(20,5){\line(1,2){20}} \put(20,5){\line(1,1){40}}
\put(40,5){\line(-1,1){40}} \put(40,5){\line(-1,2){20}}
\put(40,5){\line(0,1){40}} \put(40,5){\line(1,2){20}}
\put(60,5){\line(-3,2){60}} \put(60,5){\line(-1,1){40}}
\put(60,5){\line(-1,2){20}} \put(60,5){\line(0,1){40}}
\textcolor{blue}{\thicklines \put(0,5){\line(1,1){40}}
\put(0,5){\line(3,2){60}} \put(20,5){\line(0,1){40}}
\put(20,5){\line(1,1){40}} \put(40,5){\line(-1,1){40}}
\put(40,5){\line(1,2){20}} \put(60,5){\line(-3,2){60}}
\put(60,5){\line(-1,1){40}} \put(60,5){\line(-1,2){20}} }
\put(-27,-2.5){$u_1=u_2$} \put(-2.5,48.5){$v_1$}
\put(17.5,48.5){$v_2$} }
\end{picture}

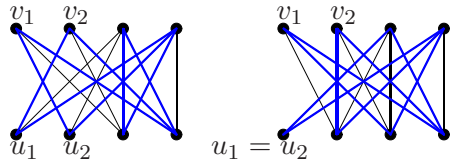
\small Figure~10. Graphs obtained from $K_{4,4}$ by removing two
edges $u_1v_1,u_2v_2$.
\end{center}

\underline{Case C. $e(G)=16$ or 17}. In this case we need find a
$C_4$-free subgraphs with $9$ edges. We only deal with the case
$e(G)=16$. Recall that $G$ is 2-connected and $\delta(G)\geq 3$.
Thus $G$ has a cycle of order at least 6. Let $C$ be a longest cycle
of $G$.

If $C$ is a Hamilton cycle, then $C$ has neither $2$-chord nor
$4$-chord, for otherwise $C$ with such a chord form a $C_4$-free
subgraph of $G$ with $9$ edges. Note that $C$ has at most eight
$3$-chords and $e(G)=16$, which implies that $C$ has all $3$-chords,
i.e., $G=K_{4,4}$. Thus $G$ contains a $\varTheta_{3,3,3}$, which is
a $C_4$-free subgraph of $G$ with $9$ edges.

Now consider the case $C=C_7$. Recall that $d_G(v)\ge 3$ for all
$v\in V(G)$. The vertex outside $C$ has three neighbors in $C$. It
follows that $C$ has a $1$-bridge or $3$-bridge. Thus $C$ together
with such a bridge form a $C_4$-free subgraph of $G$ with $9$ edges.

Finally we assume that $C=C_6$. Let $\{v_1,v_2\}=V(G)\setminus
V(C)$. Assume first that $v_1v_2\in E(G)$. Since $d_G(v)\ge 3$ for
all $v\in V(G)$, there are two distinct vertices $x_1,x_2\in V(C)$
such that $v_1x_1, v_2x_2\in E(G)$. If $x_1,x_2$ has distance $1$ or
$2$ in $C$, then $G$ has a cycle longer than $C$, a contradiction.
Thus, the distance between $x_1$ and $x_2$ in $C$ has to be $3$. In
this case, $C$ together with edges $v_1v_2, v_1x_1,v_2x_2$ form a
$\varTheta_{3,3,3}$ which is a required subgraph of $G$. {Next,
assume that $v_1v_2\notin E(G)$. For $i\in[1,2]$, $|N_G(v_i)\cap
V(C)|=3$ since $d_G(v_i)\ge 3$ and $C$ is a longest cycle of $G$.}
If $C$ has a $2$-chord, then $G$ contains a $\varTheta_{1,2,4}$. By
adding two edges incident to $v_1,v_2$, respectively, we can get a
$C_4$-free subgraph of $G$ with $9$ edges. Thus $C$ has no
$2$-chord. Note that $C$ has at most three $3$-chords. We have
$e(G)\leq 6+3+3+3<16$, a contradiction.
\end{proof}

\section{Concluding Remarks}\label{Sec-re.}
In this paper, we have studied the Tur\'an problem of directed
paths, directed cycles and all orientations of $C_4$. In particular,
for directed cycles, our results are exhaustive. Compared with the
Tur\'an problem of cycle, especially even cycles, the Tur\'an
problem of directed cycles seems to be simpler.

However, compared with the Tur\'an problem of paths in graph, we
find that the Tur\'an problem of directed paths is more complex,
because the Tur\'an number of $\overrightarrow{P_k}$ is a piecewise
function of $n$ and $k$. In extremal graph theory, this is a very
interesting difference between graphs and digraphs. In deed, for
$k+1\leq n\leq 2k-1$, we have ${\rm
ex}(n,\overrightarrow{P_{k+1}})\geq(k-1)(n-1)$ by the following
construction: Let $D$ be the digraph that (i) $V(D)$ has a partition
$\{X,Y\}$ with $|X|=k-1$, $|Y|=n-k+1$; (ii) $D[X]$ is complete and
$D[Y]$ is empty; (iii) either $X\mapsto Y$ or $Y\mapsto X$ in $D$.
It easy to verify that every $D$ is $\overrightarrow{P_{k+1}}$-free
and $a(D)=(k-1)(n-1)$.

%consider a digraph $$ which consists of the vertices $v_1,\dots, v_{k-1}, u_1,\dots, v_{n-k+1}$ together with all the arcs $v_iv_j$ ($i,j\in [1,k-1]$; $i\neq j$) and all the arcs $v_iu_j$ ($i\in [1,k-1]$; $j\in [1,n-k+1]$), that is, the digraph consists of a copy of $\overleftrightarrow{K_{k-1}}$ and $V(\overleftrightarrow{K_{k-1}})\mapsto V(L_{n,k})\setminus V(\overleftrightarrow{K_{k-1}})$.
In addition, when $2k\le n\le 3k-2$, we have the following
construction, which shows that $\overrightarrow{T_{n,k}}$ is not the
unique extremal digraph for $\overrightarrow{P_{k+1}}$ (if it is).
Let $D$ be the digraph that: (i) $V(D)$ has a partition
$\{V_1,\ldots,V_t\}$ such that $|V_i|=3$ or is even for $i\in[1,t]$;
(ii) there are exactly $n-2k$ parts of size 3; (iii) $D[V_i]$ is
empty if $|V_i|=3$, and
$D[V_i]=\overleftrightarrow{2K_{\frac{|V_i|}{2}}}$ otherwise; (iv)
$V_i\rightarrow V_j$ for $1\leq i<j\leq t$. Clearly, $D$ is
$\overrightarrow{P_{k+1}}$-free and
$a(D)=a(\overrightarrow{T_{n,k}})$. Therefore, it is extremely
difficult to determine the structure of $P_{k+1}$-free extremal
digraphs for small $n$.

By Theorem \ref{3-path}, we already know that every
$\overrightarrow{P_3}$-free extremal digraph $D$ belongs to
$\overleftrightarrow{\mathcal{T}_{n,2}}$ for $n\ge 5$. In
particular, for $\overrightarrow{P_4}$, we can deduce the following
result.

\begin{theorem}\label{p4}
For all $n\ge 9$,
\[
{\rm ex}(n,\overrightarrow{P_4})=a(\overrightarrow{T_{n,3}})=
\left\lfloor\frac{n^2}{3}\right\rfloor.
\]
%\[
%{\rm
%ex}(n,\overrightarrow{P_4})=a(\overrightarrow{T_{n,3}})=\left\lfloor\frac{n^2}{3}\right\rfloor=\left\{\begin{array}{ll}
%n^2/3,      & \mbox{if } n\equiv 0 \bmod 3;\\
%(n^2-1)/3,  & \mbox{if } n\equiv 1,2 \bmod 3.
%\end{array} \right.
%\]
Furthermore,
${\rm{EX}}(n,\overrightarrow{P_4})=\overrightarrow{\mathcal{T}_{n,3}}$.
\end{theorem}

\begin{proof}
On one hand, due to $\overrightarrow{T_{n,3}}$ is
$\overrightarrow{P_4}$-free, $a(\overrightarrow{T_{n,3}})\le {\rm
ex}(n,\overrightarrow{P_4})$; on the other hand,
since $\overrightarrow{P_4}\subset \overrightarrow{P_{1,3}}$ and
${\rm ex}(n,\overrightarrow{P_{1,3}})=a(\overrightarrow{T_{n,3}})$ for $n\ge 9$, so we have ${\rm
ex}(n,\overrightarrow{P_4})=a(\overrightarrow{T_{n,3}})$ for $n\ge
9$. Furthermore, if $D$ is a $\overrightarrow{P_4}$-free digraph of
order $n$ with $a(D)=a(\overrightarrow{T_{n,3}})$, then $D$ is also
$\overrightarrow{P_{1,3}}$-free. This implies that $D\in{\rm
EX}(n,\overrightarrow{P_{1,3}})=\overrightarrow{\mathcal{S}_n}$.
Clearly a digraph $\overrightarrow{S_n}$ is
$\overrightarrow{P_4}$-free if and only if it is in
$\overrightarrow{\mathcal{T}_{n,3}}$. we have that
${\rm{EX}}(n,\overrightarrow{P_4})=\overrightarrow{\mathcal{T}_{n,3}}$
for $n\geq 9$.
\end{proof}

By Theorem~\ref{3-path} and Theorem~\ref{p4}, we believe that each
$\overrightarrow{P_{k+1}}$-free extremal digraph belongs to
$\overrightarrow{\mathcal{T}_{n,k}}$ for $n\ge 3k$, and put forward
the following conjecture.

\begin{conjecture}
For every $k,n\in \mathbb{N}^*$ with $n\ge 3k$,
\[
{\rm ex} (n,\overrightarrow{P_{k+1}})=a(\overrightarrow{T_{n,k}}),
\]
and ${\rm
EX}(n,\overrightarrow{P_{k+1}})=\overrightarrow{\mathcal{T}_{n,k}}$.
\end{conjecture}

Furthermore, by Observation~\ref{anti-lemma}, it would be
interesting to study the Tur\'an problem of anti-directed cycles in
digraphs. It is well known that the Tur\'an problems of even cycles
in graphs or bipartite graphs have always been difficult problems in
extremal graph theory. For ${\rm ex}(n,n;C_{2k})$ with $k\ge 2$, it
is still a major open question to determine this value,
see~\cite{Robert2016The, verstraete2016extremal}. In particular, for
${\rm ex}(n,n; C_4)$, this problem is also a special case of
Zarankiewicz problem. The classical Zarankiewicz problem~\cite{Zar}
asks for the maximum number of edges in a bipartite graph that has a
given number of vertices but has no complete bipartite subgraphs of
a given size. Reiman in 1958~\cite{reiman1958problem} showed that if
$n=q^2+q+1$ where $q$ is the order of a projective plane (in
particular $q$ is a prime power), then ${\rm
ex}(n,n;C_4)=(q^2+q+1)(q+1)$. So far, for general $n$, the precise
value of ${\rm ex}(n,n;C_4)$ is still an open problem.
%On the other hand, let $B=B(X,Y; E)$ be a $C_4$-free bipartite graph
%with $|X|=|Y|=n$, $\sum_{x\in X}\binom{d_B(x)}{2}$ counts the number
%of pairs in $Y$ with a common neighbor in $X$ (with multiplicity).
%Since no such pair has two common neighbors in $B$, the sum is at
%most $\binom{n}{2}$, we have $e(B)\le n\sqrt {n-3/4}+n/2$.
%Therefore, ${\rm ex}(n,n;C_4)\sim n^{3/2}$.

The study of Tur\'an problem of anti-directed cycles may provide a
good way to solve the Tur\'an problems of even cycles in bipartite
graphs. By Theorem~\ref{anticycle}, we have already got that ${\rm
ex}(n,\overrightarrow{C_4^a})={\rm ex}(n,n;C_4)$ for $n\ge 3$.
Therefore, we put forward the following conjecture.

\begin{conjecture}
For every $k\in \mathbb{N}^*$, there exists a constant $n_0$ such
that for $n\ge n_0$,
\[
{\rm ex} (n,\overrightarrow{C_{2k}^a})={\rm ex}(n,n;C_{2k}),
\]
and ${\rm EX}(n,\overrightarrow{C_{2k}^a})=\{D: {\rm Brp}(D)\in{\rm
EX}(n,n;C_{2k})\}$.
\end{conjecture}

\section*{Acknowledgements}
The authors would like to thank Laihao Ding, Zejun Huang and Miklos Simonovits for helpful discussion on extremal digraph theory when they were visiting
the School of Mathematics at Shandong University.
%The authors would like to thank Laihao Ding, Zejun Huang and Miklos Simonovits for discussing and sharing the related papers of extremal digraphs when they were visiting
%the School of Mathematics at Shandong University.
Part of the work was done while the second author was visiting the School of Mathematics at Shandong University. He also thanks Professor Guanghui Wang for the hospitality him received.

\bibliographystyle{abbrv}

%\bibliography{ref}
%\begin{thebibliography}{10}
%
%\bibitem{Mantel1907}
%W. Mantel, Problem 28, Wiskundige Opgaven, 10:60-61, 1907.
%\vspace {-0.25cm}
%
%\bibitem{Tu1941}
%P. Tur\'{a}n, Eine Extremalaufgabe aus der Graphentheorie, Fiz Lapok, 436-452, 1941.
%\vspace {-0.25cm}
%
%\bibitem{sidorenko1995we}
%A. Sidorenko, What we know and what we do not know about Tur\'{a}n numbers, Graphs \& Combinatorics, 11(2):179-199, 1995.
%\vspace {-0.25cm}
%\end{thebibliography}

\end{document}